\documentclass[a4paper,11pt]{amsart}
\usepackage{amsfonts, amsmath, amssymb, amsthm}
\usepackage{mathrsfs, textcomp, latexsym}
\usepackage{slashed}
\usepackage{color}

\newtheorem{defn}{Definition}[section]
\newtheorem{lem}[defn]{Lemma}
\newtheorem{thm}[defn]{Theorem}
\newtheorem{cor}[defn]{Corollary}
\newtheorem{ex}[defn]{Example}

\newtheorem{prop}[defn]{Proposition}
\newtheorem{rem}[defn]{Remark}
\newtheorem{rems}[defn]{Remarks}

\def\r{\mathcal{R}}
\def\p{\mathcal{P}}
\def\i{\mathcal{I}}

\newcommand{\en}{E\ensuremath{{_0}}}

\title[Invariants for \en-semigroups on II$_1$ factors]{Invariants for \en-semigroups on II$_1$ factors}
\author[O. Margetts]{Oliver T. Margetts}
\address{Department of Mathematics and Statistics,
 Fylde College,
Lancaster University, Lancaster LA1 4YF, U.K.}
\email{o.margetts@lancaster.ac.uk}

\author[R. Srinivasan]{R. Srinivasan}
\address{Chennai Mathematical Institute, H1, SIPCOT IT Park, Kelambakkam, Siruseri 603103, India.}
\email{vasanth@cmi.ac.in}


\begin{document}


\newcommand{\norm}[1]{\ensuremath{\left\|#1\right\|}}
\newcommand{\ip}[1]{\ensuremath{\left\langle#1\right\rangle}}
\newcommand{\dist}{\hbox{dist}}
\newcommand{\bra}[1]{\ensuremath{\left\langle#1\right|}}
\newcommand{\ket}[1]{\ensuremath{\left|#1\right\rangle}}
\newcommand{\lin}{\ensuremath{\mathrm{Span}}}
\renewcommand{\ker}{\ensuremath{\mathrm{ker}}}
\newcommand{\ran}{\ensuremath{\mathrm{Ran}}}
\newcommand{\dom}{\ensuremath{\mathrm{Dom}}}
\newcommand{\supp}{\ensuremath{\mathrm{supp}}}
\newcommand{\id}{\hbox{id}}

\newcommand{\overtimes}{\ensuremath{\overline{\otimes}}}
\newcommand{\undertimes}{\ensuremath{\underline{\otimes}}}
\newcommand{\opower}[1]{\ensuremath{^{\otimes #1}}}

\newcommand{\N}{\ensuremath{\mathbb{N}}} 
\newcommand{\Z}{\ensuremath{\mathbb{Z}}} 
\newcommand{\Q}{\ensuremath{\mathbb{Q}}} 
\newcommand{\R}{\ensuremath{\mathbb{R}}} 
\newcommand{\C}{\ensuremath{\mathbb{C}}} 
\newcommand{\F}{\ensuremath{\mathcal{F}}} 
\newcommand{\B}{\ensuremath{\mathcal{B}}}
\newcommand{\D}{\ensuremath{\mathcal{D}}} 
\newcommand{\E}{\ensuremath{\mathcal{E}}} 
\newcommand{\W}{\ensuremath{\mathcal{W}}}
\newcommand{\V}{\ensuremath{\mathcal{V}}}
\renewcommand{\H}{\ensuremath{\mathcal{H}}}
\newcommand{\K}{\ensuremath{\mathsf{K}}}
\newcommand{\h}{\ensuremath{\mathrm{h}}}
\renewcommand{\k}{\ensuremath{\mathrm{k}}}
\newcommand{\J}{\ensuremath{\mathscr{J}}}
\newcommand{\A}{\ensuremath{\mathcal{A}}}
\renewcommand{\L}{\ensuremath{\mathcal{L}}}
\newcommand{\n}{\ensuremath{\mathrm{N}}} 
\newcommand{\m}{\ensuremath{\mathrm{M}}} 

\newcommand{\munit}{\ensuremath{\mu\hbox{nit~}}}
\newcommand{\munits}{\ensuremath{\mu\hbox{nits~}}}
\newcommand{\unitset}[1]{\ensuremath{\mathcal{U}_{#1}}}
\newcommand{\unitspace}[1]{\ensuremath{H(\mathcal{U}_{#1})}}
\newcommand{\munitset}[1]{\ensuremath{\mathcal{U}_{#1,#1'}}}
\newcommand{\munitspace}[1]{\ensuremath{H(\mathcal{U}_{#1,#1'})}}
\newcommand{\gaugespace}[1]{\ensuremath{H(G(#1))}}
\newcommand{\ind}{\ensuremath{\mathrm{Ind}}}

\newcommand{\expectation}{\ensuremath{\mathbb{E}}}
\newcommand{\Exp}{\ensuremath{\mathrm{Exp}}}
\newcommand{\Log}{\ensuremath{\mathrm{Log}}}

\newcommand{\alg}{\ensuremath{\mathrm{A}}}
\newcommand{\factor}{\ensuremath{\mathrm{M}}}
\newcommand{\oneinfty}{I$_\infty~$}
\newcommand{\twoone}{II$_1~$}
\newcommand{\twoinfty}{II$_\infty~$}
\newcommand{\three}[1]{III$_{#1}~$}
\newcommand{\hyperfinite}{\ensuremath{\mathcal{R}}}
\newcommand{\semiflowalg}{\ensuremath{\mathcal{A}}}

\newenvironment{makered}{\color{red}}{}

 \dedicatory{We humbly dedicate this paper to the memory of Bill Arveson.}

  \begin{abstract}
 We introduce four new cocycle conjugacy invariants for $E_0$-semigroups on II$_1$ factors: a coupling index, a dimension for the gauge group, a \emph{super product system} and a $C^*$-semiflow. Using noncommutative It\^o integrals we show that the dimension of the gauge group can be computed from the structure of the \emph{additive cocycles}. We do this for the Clifford flows and even Clifford flows on the hyperfinite \twoone factor, and for the free flows on the free group factor $L(F_\infty)$. In all cases the index is $0$, which implies they have trivial gauge groups. We compute the super product systems for these families and, using this, we show they have trivial coupling index. Finally, using the $C^*$-semiflow and the boundary representation of Powers and Alevras, we show that the families of Clifford flows and even Clifford flows contain infinitely many mutually non-cocycle-conjugate \en-semigroups.
 \end{abstract}

 \maketitle

\section{Introduction}
  A weak-* continuous semigroup of unital $*$-endomorphisms on a von Neumann algebra is called an \en-semigroup. They arise naturally in the study of open quantum systems (\cite{ccr}, \cite{dilations}), the theory of interactions (\cite{absorbing states}, \cite{interactions}, \cite{arveson}), and in algebraic quantum field theory (simply restrict the time evolution to an algebra of observables corresponding to the future light cone).
  For \en-semigroups on type I factors the subject has grown rapidly since its inception in \cite{powers} (see the monograph \cite{arveson} for extensive references). Arveson showed that these \en-semigroups are completely classified by continuous tensor products of Hilbert spaces, called \emph{product systems}, and this gives a rough division into ``types'' I, II and III. The type I \en-semigroups on type I factors are just the CCR flows (\cite{arveson}), but there are uncountably many exotic product systems of types II and III (\cite{genccr}, \cite{toepcar} \cite{lieb}, \cite{tsirel}).
  
  By contrast, after their study was initiated by Powers (in the same 1988 paper \cite{powers}), there has been little progress regarding \en-semigroups on type II$_1$ factors. 
  In \cite{alev} Alexis Alevras made developments for \en-semigroups on \twoone factors analogous to the theory on type I factors. He associated a product system of Hilbert modules to every \en-semigroup and showed that they form a complete invariant (product systems of Hilbert modules have also been considered in \cite{tips}, \cite{dilations}, but they are slightly different in form and function to the ones considered here). He also introduced an index using Powers' boundary representation (\cite{powers}) and computed the index for several important cases.  
  
  Still, this does not classify even the simplest examples of \en-semigroups on the hyperfinite II$_1$ factor. One problem is that Alevras was unable to show his index is an invariant up to cocycle conjugacy (see Section 2). For type I factors, Powers showed his index is a cocycle conjugacy invariant by proving it equals the Arveson index, an intrinsic property of the product system (\cite{powersrob}, \cite{powersprice}). For \twoone factors there is no known connection between the index of the semigroup and the product system of Hilbert modules. Directly related to this is a lack of effective invariants for these objects.
    
    The structure of the paper is as follows. In Section 2 we give the basic definitions of \en-semigroups, cocycle conjugacy, units and the gauge group. We introduce three important families of examples: Clifford flows and even Clifford flows on the hyperfinite \twoone factor, and free flows on the free group factor $L(F_\infty)$. All three are, in a generalized sense, second quantizations of unilateral shifts on $L^2$ spaces.
  
  In Section 3, for an \en-semigroup $\alpha$ on $\m$, we use use the antilinear *-isomorphism $j=Ad(J):\m\to\m'$ of Tomita-Takesaki theory to define a complementary \en-semigroup $\alpha'$ on $\m'$. We associate a Hilbert space to the pair $(\alpha,\alpha')$ and the dimension of this Hilbert space is a cocycle conjugacy invariant, called the coupling index. If there exists an \en-semigroup $\sigma$ on $B(H)$ extending both $\alpha$ and $\alpha'$ then the coupling index is equal to the usual Powers-Arveson index of $\sigma$. It should be noted that this index is not same as the index defined by Alevras in \cite{alev} (they differ, for instance, on Clifford flows and even Clifford flows). We finish the section by associating a Hilbert space $H(G(\alpha))$ to the gauge group $G(\alpha)$ and showing the dimension of this Hilbert space is another cocycle conjugacy invariant.
  
  In section 4 we define \emph{super product systems}, generalisations of Arveson's product systems of Hilbert spaces. These are similar in spirit to the subproduct systems already studied in \cite{BM}, \cite{OB} and play a prominent role in the paper. We introduce multiplicative and additive units for super product systems and, in section 5, we prove that there is a one to one correspondence between the two. To this end we develop a noncommutative stochastic calculus for super product systems, very similar to that of \cite{bernoulli}. This gives an explicit formula for computing the index of a super product system from its additive units. 
  
  In section 6 we study the gauge group from the point of view of noncommutative probability. It follows from the invariance of the trace that an element $(U_t)_{t\geq0}$ of the gauge group enjoys the future independence property $\tau(U_t\alpha_t(x))=\tau(U_t)\tau(\alpha_t(x))$ for all $x\in\factor$, $t\geq0$, hence $G(\alpha)$ generates a noncommutative white noise, similar to those of \cite{kostler}, \cite{bernoulli}. By recasting the white noise as a super product system we see the dimension of the gauge group can be computed using the methods of section 5. Furthermore, $\dim H(G(\alpha))$ is zero precisely when the gauge group is isomorphic to $(\R,+)$.

  In section 7 we compute the additive cocycles for the Clifford flows, even Clifford flows, and free flows explicitly. Using the results of section 6 we show that the gauge group is trivial for all these examples. By a result of Arveson, this shows that a one-parameter group of automorphisms on $B(H)$ extending one of these semigroups is completely determined by its ``past'' and ``future'' \en-semigroups on $\m'$, respectively $\m$ (see \cite{interactions}, \cite{arveson}).

  The pair $(\alpha,\alpha')$ gives us a pair of product systems of Hilbert modules. In Section 8, we show their intersection is a super product system. We prove this is an invariant for $\alpha$ and compute it for Clifford flows and even Clifford flows. This has strong structural implications. For instance, it precludes the existence of certain extensions of the Clifford, or even Clifford flows, to $B(H)$ and allows us show the coupling index is $0$ for all the Clifford flows and even Clifford flows. For free flows we show that the super product system is trivial and hence free flows also have coupling index $0$.
  
  In section 9 we introduce the notion of a $\tau$-semiflow, the natural counterpart to the $C^*$-semiflows introduced by Floricel for type I factors (\cite{remus}). We show that the $\tau$-semiflow is a cocycle conjugacy invariant and use this to show that, for a certain class of \en-semigroups, cocycle conjugacy is equivalent to conjugacy. Using this, together with the computation of Powers-Alevras index for even Clifford flows, we are able to prove that the even Clifford flows are not cocycle conjugate among themselves for different ranks. We lift this result to the Clifford flows with a simple argument. 

\section{Preliminaries}\label{pre}

\begin{defn}
An \en-semigroup on the von Neumann algebra $\m$ is a semigroup $(\alpha_t)_{t\geq0}$ of normal, unital *-endomorphisms of $\m$ satisfying
\begin{itemize}
 \item[(i)] $\alpha_0=id$,
 \item[(ii)] $\alpha_t(\factor)\neq\factor$ for all $t\geq0$,
 \item[(iii)] $t\mapsto \rho(\alpha_t(x))$ is continuous for all $x\in\factor$, $\rho\in\factor_*$.
\end{itemize}
\end{defn}

\begin{defn} A cocycle for an \en-semigroup $\alpha$ on $\m$ is a strongly continuous family of unitaries $U=(U_t)_{t\geq0}$ satisfying $U_s\alpha_s(U_t)=U_{s+t}$ for all $s,t\geq0$.
\end{defn}

For a cocycle $U$, we automatically have $U_0=1$. Furthermore the family of endomorphisms $\alpha_t^U(x):=U_t\alpha_t(x)U_t^*$ defines an \en-semigroup. This leads to the following equivalence relations on \en-semigroups.

\begin{defn}\label{conjugacy def} Let $\alpha$ and $\beta$ be \en-semigroups on von Neumann algebras $\m$ and $\n$. 
 \begin{itemize}
  \item[(i)] $\alpha$ and $\beta$ are \emph{conjugate} if there exists a *-isomorphism $\theta:\m\to\n$ such that $\beta_t=\theta\circ\alpha_t\circ\theta^{-1}$ for all $t\geq0$.
  \item[(ii)] $\alpha$ and $\beta$ are \emph{cocycle conjugate} if there exists a cocycle $U$ for $\alpha$ such that $\beta$ is conjugate to $\alpha^U$. 
 \end{itemize}
\end{defn}

Let $\m$ be a von Neumann algebra acting standardly and $\pi:\m\to B(H)$ a normal representation. Then for an \en-semigroup $\alpha$ on $\pi(\m)$ there exists a conjugate semigroup $\pi^{-1}\circ\alpha\circ\pi$ on $\m$. Thus with out loss of generality we may restrict to algebras acting standardly.

Let $\alpha$, $\beta$ be \en-semigroups acting on \twoone factors $\m$ and $\n$ and suppose $\theta:\m\to\n$ is a *-isomorphism intertwining $\alpha$ and $\beta$. By uniqueness of the trace on $\m$ we have $\tau_\n\circ\theta=\tau_\m$, hence $\theta$ extends to a unitary $U:L^2(\m)\to L^2(\n)$. This unitary satisfies
\begin{itemize}
 \item[(i)] $U \Omega_\m=\Omega_\n$,
 \item[(ii)] $U\m U^*=\n$,
 \item[(iii)] $\beta_t(x)=U\alpha_t(U^*xU)U^*$ for all $t\geq0$, $x\in\n$,
\end{itemize} where $\Omega_M$(respectively $\Omega_N$) is unique cyclic and separating vector (which evaluates the trace), given by the image of $1_M$ (respectively $1_N$) in $L^2(M)$ (respectively $L^2(N)$).
For \twoone factors acting standardly we take this as the definition of conjugacy.

For the rest of the paper $\m$ will denote a II$_1$ factor with trace $\tau$ acting canonically on $H:=L^2(\m)$ with cyclic and separating trace vector $\Omega$. Let $\alpha$ be an E$_0$-semigroup on $\m$. Associated to $\alpha$ is a decreasing family of von Neumann algebras
$$\m=\alpha_0(\m) \supset \alpha_s(\m)\supset\alpha_t(\m)\supset \C1 \quad\hbox{for~all}\quad 0\leq s\leq t.$$
By taking relative commutants we get a filtration
$$\C1=\m_0 \subset \m_s\subset\m_t\subset \m \quad\hbox{for~all}\quad 0\leq s\leq t,$$
where $\m_t:=\m\cap\alpha_t(\m)'$.

Using the cyclic and separating property of $\Omega$ we can well-define a linear operator on the subspace $\factor\Omega$ by
$$ S_t x\Omega:= \alpha_t(x)\Omega \quad\hbox{for~all}\quad x\in\m.$$
By invariance of the trace under $\alpha_t$, each $S_t$ extends to an isometry and it is easily seen that $\{S_t:t\geq 0\}$ is a strongly continuous semigroup satisfying $S_tx=\alpha_t(x)S_t$. This motivates the following definition.

\begin{defn} A unit for $\alpha$ is a strongly continuous semigroup $T=\{T_t:t\geq 0\}$ of operators in $B(H)$ such that $T_0=1$ and $T_tx=\alpha_t(x)T_t$ for all $t\geq0$, $x\in\factor$. Denote the collection of units by $\mathcal{U}_\alpha$. We call $\{S_t:t\geq 0\}$ the \emph{canonical unit} associated to $\alpha$. \end{defn}

A gauge cocycle for $\alpha$ is a cocycle which satisfies the adaptedness condition $U_t\in\m_t$ for all $t\geq 0$. Notice that the product of two gauge cocycles $U$ and $V$ is another gauge cocycle and that the adjoint $U^*$ of a gauge cocycle is a gauge cocycle. Thus, under the multiplication $(UV)_t:=U_tV_t$, the collection of all gauge cocycles forms a group, denoted by $G(\alpha)$, called the gauge group of $\alpha$.

\begin{prop}\label{actionprop} The multiplication $(U,T)\mapsto UT$ defines an action of the gauge group on $\mathcal{U}_\alpha$.
\end{prop}

\begin{proof} The adaptedness condition ensures that
$$U_tT_tx=U_t\alpha_t(x)T_t=\alpha_t(x)U_tT_t \qquad (t\geq0,~x\in\m),$$
and it follows from the cocycle identity that
 $$U_sT_sU_tT_t=U_s\alpha_s(U_t)T_sT_t=U_{s+t}T_{s+t} \qquad (s,t\geq0).$$
Moreover, for $\xi\in H$, we have
\begin{eqnarray*}\norm{(U_tT_t-I)\xi}^2&=&\norm{U_tT_t\xi}^2+\norm{\xi}^2-2\Re\ip{\xi,U_tT_t\xi}\\
 &=&\norm{T_t\xi}^2+\norm{\xi}^2-2\Re\ip{U_t^*\xi,T_t\xi}\to 0
\end{eqnarray*}
 as $t\to 0$. Thus, since $U_0=1$, $UT$ is a unit.   \end{proof}

By Proposition \ref{actionprop}, the canonical unit induces a map $$G(\alpha)\to \mathcal{U}_\alpha, ~U\mapsto US.$$ If $US=VS$ for two gauge cocycles $U,V$ then $(U_t-V_t)S_t=0$ for all $t\geq0$, so that $(U_t-V_t)\Omega=0$. As $\tau$ is faithful we obtain $U_t-V_t=0$ for all $t\geq0$, so the canonical map $G(\alpha)\to\mathcal{U}_\alpha$ is an injection.

We end this section with some examples and by fixing some notation. These examples are already discussed in \cite{alev}.

Throughout this paper, we denote by $\k$ an arbitrary separable real Hilbert space with dimension equal to $n \in \{1,2 \cdots \infty\}$, and by $\k^\C$ be the complexification of $\k$.

Let $L^2(\R_+; \k)$ or $L^2(\R_+; \k^\C)$ be the Hilbert space of square integrable functions taking values in $\k$ or $\k^\C$ respectively. Let $\{T_t\}$ be the shift semigroup of $L^2(\R_+;\k^\C)$ defined by 
\begin{eqnarray*}(T_tf)(s) & = & 0, \quad s<t,\\
& = & f(s-t), \quad s \geq t,
\end{eqnarray*} 
for $f \in L^2(\R_+;\k^\C).$ They are semigroups of isometries and we denote their restriction to $L^2(\R_+; \k)$ also by $\{T_t\}$.

The full Fock space is defined by $$\Gamma_f(L^2(\R_+; \k^\C))=\bigoplus_{n=0}^\infty L^2(\R_+; \k^\C)\opower{n},$$ where $L^2(\R_+; \k^\C)\opower{0}=\C\Omega$, and $\Omega$ will be called as vacuum vector. 

\begin{ex}\label{clifford}
{\bf Clifford flows} Let $\H$ be a real Hilbert space and $\H_\C$ its complexification. Write $\Gamma_a(\H_\C)$ for the antisymmetric Fock space over $\H_\C$, i.e. the subspace of $\Gamma_f(\H_\C)$ generated by antisymmetric tensors. For any $f\in \H_\C$ the Fermionic creation operator $a^*(f)$ is the bounded operator defined by the linear extension of
$$a^*(f)\xi=\left\{\begin{array}{ll} f & \hbox{if}~\xi=\Omega\\ f\wedge \xi &  \hbox{if}~ \xi\perp\Omega, \end{array}\right.$$
where $f\wedge\xi$ is the image of $f\otimes\xi\in \Gamma_f(\H_\C)$ under orthogonal projection onto $\Gamma_a(\H_\C)$. The annihilation operator is defined by $a(f)=a^*(f)^*$.
The unital $C^*$-algebra $Cl(\H)$ generated by the self-adjoint elements $$\{u(f)=(a(f)+a^*(f))/\sqrt{2}:~f\in\H\}$$ is the Clifford algebra over $\H$. The vacuum $\Omega$ is cyclic and defines a tracial state for $Cl(\H)$, so the weak completion yields a II$_1$ factor; in fact it is the hyperfinite II$_1$ factor $\mathcal{R}$ \cite{voiculescu}.

If $\H=L^2(\R_+;\k)$, where $\k$ is a separable Hilbert space with dimension $n\in \{1,2,\cdots \infty\}$ as mentioned before, then $T$
the unilateral shift on $L^2(\R_+;\k)$ defines an \en-semigroup on $\mathcal{R}$ by extension of $$\alpha^n_t(u(f_1)\cdots u(f_k))=u(T_tf_1)\cdots u(T_tf_k)$$ called the Clifford flow of rank $n$.
\end{ex}

\begin{ex}\label{evenclifford}{\bf Even Clifford flows} The von Neumann algebra generated by the even products $$\r_e=\{u(f_1)u(f_2)\cdots u(f_{2n}): f_i \in L^2((0,\infty), \k),~n \in \N\}$$ is also isomorphic to the hyperfinite $II_1$ factor. The restriction of the Clifford flow $\alpha^n$ of rank $n$ to this subfactor is called the even Clifford flow of rank $n$. We denote it by $\beta^n$.

\end{ex}

\begin{ex}\label{free}{\bf Free flows}

Let $\k$ be a real Hilbert space of dimension $n \in \{1,2,\cdots \infty\}$ and for every $f\in L^2(\R_+;\k)$ define the operator $s(f)=\frac{l(f)+l(f)^*}{2}$ on $\Gamma_f(L^2(\R_+;\k^\C))$ where
$$l(f)\xi=\left\{\begin{array}{ll} f& \hbox{if}~\xi=\Omega, \\ f\otimes \xi & \hbox{if}~\ip{\xi,\Omega}=0. \end{array}\right.$$
The von Neumann algebra  $\Phi(\k)=\{s(f):f\in L^2(\R_+;\k)\}'' $, is isomorphic to the free group factor $L(F_\infty)$ and the vacuum is cyclic and separating with $\ip{\Omega,x\Omega}=\tau(x)$ (see \cite{voiculescu}).

Let $T$ be the unilateral shift on $L^2(\R_+;\k)$. Then there exists a unique \en-semigroup $\gamma^n$ on $\Phi(\k)$ satisfying
 $$\gamma^n_t(s(f)):=s(T_tf) \qquad (f\in\k,~t\geq0)$$
(see \cite{alev}), this is called the free flow of multiplicity $\dim\k$.
\end{ex}

\medskip

{\bf Notation:} Throughout this paper, for $E \subset B(H)$, we shall write $\left[E\right]$ for the closure, in the weak operator topology, of the linear subspace of $B(H)$ spanned by
$E$. Similarly, if $S \subset H$ is a subset of vectors, we shall
write $\left[S\right]$ for the norm-closed subspace of $H$ spanned
by $S$.  $\N$ denotes the set of natural numbers and $\N_0=\N\cup \{0\}.$

\section{Multi-units, index and gauge dimension}\label{multi-units}

We begin this section by defining a complementary, or dual \en-semigroup. Let $J$ be the modular conjugation associated to the vector $\Omega$ by the Tomita-Takesaki theory. We can define a complementary \en-semigroup on $\m'$ by setting
$$\alpha'_t(x')=J\alpha_t(Jx'J)J\qquad (x'\in\m').$$
Then we define a semigroup of maps on $\m\cup\m'$ by setting
$$\mu_t(x)=\left\{\begin{array}{ll} \alpha_t(x), & x\in\m, \\ \alpha'_t(x), & x\in\m'.\end{array}\right.$$

\begin{prop}
If the \en-semigroups $\alpha$ and $\beta$ on $\m$ are cocycle conjugate, the complementary \en-semigroups are also cocycle conjugate.
\end{prop}

\begin{proof}
It is easy to see conjugacy is preserved under complementation. Suppose $\{U_t\}$ is an $\alpha-$cocycle and $\beta_t(\cdot)=U_t\alpha(\cdot)U_t^*$. For any $t \geq 0$, let $v_t=JU_tJ$, then $v_t \in \m'$ and $v_t$ satisfies $$v_{s+t}=Ju_{s+t}J=JU_sJJ\alpha_s(U_t)J= JU_sJ\alpha_s'(JU_tJ)=v_s\alpha_s'(v_t).$$ So $\{v_t\}$ forms an $\alpha'$-cocycle. We also have \begin{align*}\beta_t'(m') &=J\beta_t(Jm'J)J\\ &=JU_t\alpha(Jm'J)U_t^*J\\ &=(JU_tJ)(J\alpha(Jm'J)J)(JU_t^*J)\\ &=v_t\alpha_t'(m')v_t^*,\\ \end{align*} for all $m'\in \m'$. 
  \end{proof}

\begin{defn} A \munit or multi-unit for the \en-semigroup $\alpha$ is a stongly continuous semigroup of bounded operators $(T_t)_{t\geq0}$ satisfying
$$T_tx=\mu_t(x)T_t \quad\hbox{for~all}\quad x\in\m\cup\m'$$ 
together with $T_0=1$. That is, a multi-unit is a unit for both $\alpha$ and $\alpha'$. Denote the collection of \munits for $\alpha$ by $\munitset{\alpha}$.
\end{defn}

We have already noticed $\{S_t:t\geq 0\}$ is a unit for $\alpha$. Since the \en-semigroup $\alpha$ is $*-$preserving, the modular conjugation operator commutes with $S_t$ for each $t\geq 0$. This implies $$\alpha'_t(m')S_t=J\alpha_t(Jm'J)JS_t=JS_t(Jm'J)J=S_tm' $$ for all $m'\in \m', t\geq 0,$ and $S\in \munitset{\alpha}$.   So the collection of multi-units for any \en-semigroup in a II$_1$ factor is always non-empty.

\begin{lem} If $U$ is a gauge cocycle for $\alpha$ then $US$ is a \munit for $\alpha$. Thus $U\mapsto US$ defines an injection $G(\alpha)\rightarrow\munitset{\alpha}$. \end{lem}

\begin{proof} Since $US$ is a unit for $\alpha$, we need only show that it is also a unit for $\alpha'$. For all $t\geq0$, $x\in\m'$ and $y\in\m$ we have
\begin{align*}\alpha_t'(x)U_tS_ty\Omega&=\alpha_t'(x)U_t\alpha_t(y)\Omega=U_t\alpha_t(y)\alpha_t'(x)\Omega\\
&=U_t\alpha_t(y)J\alpha_t(JxJ)J\Omega
=U_t\alpha_t(y)\alpha_t(Jx^*J)\Omega\\
&=U_tS_tyJx^*J\Omega=U_tS_tyx\Omega=U_tS_txy\Omega.
\end{align*}   \end{proof}

\begin{lem} Let $X$ and $Y$ be \munits for $\alpha$. Then $X_t^*Y_t=e^{\lambda t}1$ for some constant $\lambda\in \C$.\end{lem}

\begin{proof} For any $x\in\m\cup\m'$ we have
$$X^*_tY_tx=X_t^*\mu_t(x)Y_t=(\mu_t(x^*)X_t)^*Y_t=(X_tx^*)^*Y_t=xX_t^*Y_t,$$
hence $X_t^*Y_t\in(\m\cup\m')'= \C 1$. We note that $$X_s^*Y_sX_t^*Y_t=X_t^*X_s^*Y_sY_t=X_{s+t}^*Y_{s+t}$$ and hence the complex valued function $f(t)=\ip{\Omega,X_t^*Y_t\Omega}$ is continuous and satisfies $f(s+t)=f(s)f(t)$. Since $f(0)=1$ we have $f(t)=e^{\lambda t}$ for some $\lambda\in\C$.   \end{proof}

Thus we can define a covariance function $c:\mathcal{U}_{\alpha,\alpha'}\times\mathcal{U}_{\alpha,\alpha'}\to\C$ by $X_t^*Y_t=e^{c(X,Y) t}1$ for all $t\in\R_+$. Since the covariance function is conditionally positive definite (see Proposition 2.5.2 of \cite{arveson}) the asignment
$$ \ip{f,g}\mapsto \sum_{X,Y\in \mathcal{U}_{\alpha,\alpha'}}{c(X,Y)\overline{f(X)}g(Y)} $$
defines a positive semidefinite form on the space of finitely supported functions $f:\mathcal{U}_{\alpha,\alpha'}\to\C$ satisfying $\sum_{X\in \mathcal{U}_{\alpha,\alpha'}}{f(X)}=0$. Hence we may quotient and complete to obtain a Hilbert space $H(\mathcal{U}_{\alpha,\alpha'})$. 

\begin{defn} Define the coupling index $\ind_c(\alpha)$ of the \en-semigroup $\alpha$ as the cardinal $\dim H(\mathcal{U}_{\alpha,\alpha'})$.
\end{defn}

\begin{prop}\label{indexprop} If $\alpha$ and $\beta$ are cocycle conjugate \en-semigroups then they have the same coupling index. Furthermore, if $\gamma$ is an \en-semigroup on the \twoone factor $\mathrm{W}$ then $$\ind_c(\alpha\otimes\gamma)\geq \ind_c(\alpha)+\ind_c(\gamma).$$ \end{prop}

\begin{proof} For the first statement it is enough to give a bijection $\mathcal{U}_{\alpha,\alpha'}\to\mathcal{U}_{\beta,\beta'}$ preserving the covariance. For conjugate semigroups with intertwining unitary $V$, $Ad_V$ clearly does the job. So assume $\alpha=\beta^U$, for an $\alpha$-cocycle $U$. Then the map $T\mapsto JUJUT$ suffices (see Theorem \ref{sps invariant}). 

 For the inequality note that every pair of \munits $X^\alpha$, $X^\gamma$ for $\alpha$ and $\gamma$ respectively give a \munit $X^\alpha\otimes X^\gamma$ for $\alpha\otimes\gamma$. As
$$({X_t^\alpha}\otimes {X_t^\gamma})^*({Y_t}^\alpha\otimes {Y_t^\gamma})=e^{(c(X^\alpha,Y^\alpha)+c(X^\gamma,Y^\gamma))t}1$$
there exists an isometry $$\munitspace{\alpha}\oplus \munitspace{\gamma}\hookrightarrow H(\mathcal{U}_{\alpha\otimes\gamma,(\alpha\otimes\gamma)'})$$ (see \cite{arveson} Lemma 3.7.5).   \end{proof}

\begin{prop}\label{extn} Let $\alpha$ be an \en-semigroup on the II$_1$ factor $\factor$. If there exists an \en-semigroup $\sigma$ on $B(L^2(\factor))$ satisfying 
 $$\sigma_t(x)=\left\{\begin{array}{ll} \alpha_t(x) & \text{if }x\in\m, \\ \alpha_t'(x) & \text{if }x\in\m', \end{array}\right.\qquad\text{for all }t\geq0,$$
 then $\ind_c(\alpha)$ is equal to the Powers-Arveson index of $\sigma$.
\end{prop}

\begin{proof} 
 Clearly if $T$ is a unit for $\sigma$ then it is a \munit for $\alpha$. Conversely if $T$ is a \munit for $\alpha$ and $x\in B(H)$ we may pick nets $(y_i)_{i\in I}\subset \m$ and $(z_i)_{i\in I}\subset \m'$ satisfying $y_iz_i\to x$ in the ultraweak topology. Then by ultraweak continuity of $\sigma$,
 $\sigma_t(x)T_t=\lim_{i\in I}\sigma_t(y_iz_i)T_t$, but $$\sigma_t(y_iz_i)T_t=\alpha_t(y_i)\alpha_t'(z_i)T_t=T_ty_iz_i\to T_tx,$$
 that is, $T$ is a unit for $\sigma$. Thus $\mathcal{U}_{\alpha,\alpha'}=\mathcal{U}_\sigma$ and the induced covariance function on $\mathcal{U}_\sigma$ is precisely that of \cite{arveson}, section 2.5.   \end{proof}

\begin{rem}
All our known examples of \en-semigroups on II$_1$ factors do not admit an extension as described in Proposition \ref{extn}. This follows from our computations on the super product systems associated to these examples, in section \ref{coupling}. It is an interesting open question to construct an \en-semigroup on a II$_1$ factor which admits such an extension to $B(L^2(\cdot))$.
\end{rem}

The gauge group also forms a cocycle conjugacy invariant for $\alpha$. Conjugate \en-semigroups clearly have isomorphic gauge groups, and if $U$ is a unitary cocycle for $\alpha$ and $\beta=\alpha^U$ then we have the group isomorphism $Ad_U:G(\alpha)\to G(\beta)$.

\begin{lem} If $U,V\in G(\alpha)$ then there exists $\lambda\in\C$ such that $\tau(U_t^*V_t)=e^{\lambda t}$ for all $t\geq0$. In particular we have the identity
\begin{equation}\label{multiplicativeeq} \tau(U_{s+t}^*V_{s+t})=\tau(U_s^*V_s)\tau(U_t^*V_t) \qquad (s,t\geq0).\end{equation} \end{lem}

\begin{proof} As $\Omega$ is invariant under $S$ we have
\begin{eqnarray}\tau(U^*_tV_t)&=&\ip{\Omega,U^*_tV_t\Omega}
=\ip{\Omega,S_t^*U^*_tV_tS_t\Omega}
=e^{c(US,VS)t}\label{fundamentalidentity}\end{eqnarray}
as required. Equation (\ref{multiplicativeeq}) follows immediately.   \end{proof}

For a pair of gauge cocycles $U,V$, we will write the corresponding covariance function on $G(\alpha)$ as $c_*(U,V):=c(US,VS)$. It is clear that $c_*(\cdot,\cdot)$ is just the pullback of $c(\cdot,\cdot)$ along the injective map $G(\alpha)\to\mathcal{U}_{\alpha,\alpha'}$ induced by the canoncial unit. Equation (\ref{fundamentalidentity}) thus reduces to $\tau(U_t^*V_t)=e^{c_*(U,V)t}$. We can now repeat the above construction with $(G(\alpha),c_*(\cdot,\cdot))$ in place of $(\mathcal{U}_{\alpha,\alpha'},c(\cdot,\cdot))$ to obtain a Hilbert space $H(G(\alpha))$.

\begin{defn} The dimension of the gauge group $G(\alpha)$ is defined to be the cardinal $\dim G(\alpha):=\dim H(G(\alpha))$. We will also refer to this as the gauge index of the \en-semigroup $\alpha$. \end{defn}

\begin{thm} The gauge index is a cocycle conjugacy invariant. Furthermore, it satisfies
$$\dim G(\alpha\otimes\beta)\geq \dim G(\alpha)+\dim G(\beta)$$
for any \en-semgiroup $\beta$ on a second II$_1$ factor. \end{thm}

\begin{proof} If $U$ is a unitary cocycle for $\alpha$ then for $V,W\in G(\alpha)$,
$$\tau(U_sV_s^*U_s^*U_sW_sU_s^*)=\tau(U_sV_s^*W_sU_s^*)=\tau(V_s^*W_s),$$
so $\dim H(G(\alpha))=\dim H(G(\alpha^U))$. If $\beta$ is conjugate to $\alpha^U$, then as the intertwining *-isomorphism also preserves the trace $\dim H(G(\alpha^U))$ is equal to $\dim H(G(\beta))$. The inequality follows as in Proposition \ref{indexprop}.   \end{proof}

\section{Super product systems}

In this section and following section we develop tools to analyse semigroups on II$_1$ factors. Here we introduce the notion of a super product system of Hilbert spaces, which is a generalization of the product systems introduced by Arveson. The second named author heard this notion from C. K\"ostler during a conversation, but we could not find any literature dealing with `super product systems'. 

\begin{defn}
\label{superproductsystem}
A super product system of Hilbert spaces is a one parameter family
of separable Hilbert spaces $\{H_t: t \geq 0 \}$, together with isometries $$U_{s,t} : H_s \otimes H_t ~\mapsto H_{s+t}~ \mbox{for}~ s, t \in (0,\infty),$$
satisfying the following two axioms of associativity and measurability.

\medskip
\noindent(i) (Associativity) For any $s_1, s_2, s_3 \in (0,\infty)$
$$U_{s_1, s_2 + s_3}( 1_{H_{s_1}} \otimes U_{s_2 ,
s_3})=  U_{s_1+ s_2 , s_3}( U_{s_1 ,
s_2} \otimes 1_{H_{s_3}}).$$

\noindent (ii) (Measurability) The space $\H=\{(t, \xi_t): t \in (0,\infty), \xi_t \in H_t\}$ is equipped with a structure of standard Borel space that is compatible with the projection $p:\H\mapsto (0,\infty)$ given by $p((t, \xi_t)=t$, tensor products and the inner products  as in the definition of a product system (see 3.1.2, \cite{arveson}).  
\end{defn}

\begin{rems}
The theory of subproduct systems or inclusion systems is studied in \cite{BM} and \cite{OB}, where the embedding map is reversed, that is the operator $U_{s,t} : H_s \otimes H_t ~\mapsto H_{s+t}$ is assumed to be a co-isometry. But unlike subproduct systems, which can possibly be finite dimensional, super product systems are either all one-dimensional or all infinite dimensional. If $d(t)$ is the dimension of the separable Hilbert space $H_t$, then, since $ H_s \otimes H_t$ embeds isometrically into $H_{s+t}$, the relation $d(s)d(t)\leq d(s+t)$ is satisfied for all $s,t>0$. The only possibility is $d(t)=1$ for all $t>0$ or $d(t)=\infty$ for all $t>0$.
\end{rems}

\begin{defn} By an isomorphism between super product systems $(H^1_t, U^1_{s,t})$ and $(H^2_t, U^2_{s,t})$ we mean an isomorphism of Borel spaces $V:\H^1 \mapsto \H^2$ whose restriction to each fiber provides an unitary operator $V_t:H^1_t\mapsto H^2_t$ satisfying \begin{equation}\label{prodiso}V_{s+t}U^1_{s,t}= U_{s,t}^2 (V_s \otimes V_t).\end{equation} 

\end{defn}

\begin{defn} A unit for a super product system $(H_t, U_{s,t})$ is a measurable section $\{u_t :u_t \in H_t\}$ satisfying $$U_{s,t}(u_s\otimes u_t)=u_{s+t}~\forall ~ s,t \in (0,\infty).$$ \end{defn}

Similar to product systems, a super product system is called spatial if it admits a unit. From here onwards we assume all our super product systems admit a unit, since that is the kind of super product system we will encounter while dealing with II$_1$ factors. We fix a unit denoted by $\{\Omega_t\in H_t\}$ and call that the canonical unit. We set $H_0=\C\Omega_0$.

\begin{defn} Let  $(H_t, U_{s,t})$ be a spatial super product system with the canonical unit $\{\Omega_t\}$. A unital unit is a unit $\{u_t\}_{t\geq0}$ satisfying $\norm{u_t}=1$. An exponential unit is a unit $\{u_t\}_{t\geq0}$ satisfying $\ip{u_t,\Omega_t}=1$. 
We denote the set of unital units by $\mathfrak{U}(H)$ and the collection of exponential units by $\mathfrak{U}_\Omega(H)$.
\end{defn}

\begin{rems} Every unit $\{u_t\}$ satisfies
$$\norm{u_{s+t}}=\norm{u_s\otimes u_t}=\norm{u_s}\norm{u_t} \forall ~s,t \geq 0.$$
So $\norm{u_t}=e^{\lambda t}$ for some $\lambda\in\R$. On the other hand every unit also satisfies 
$$\ip{u_{s+t}, \Omega_{s+t}}=\ip{u_s\otimes u_t, \Omega_{s}\otimes \Omega_t}= \ip{u_{s}, \Omega_{s}} \ip{u_{t}, \Omega_{t}}~~ \forall s,t \geq 0.$$ So $\ip{u_t,\Omega_t}=e^{\mu t}$ for some $\mu\in\R$. 
Thus given an exponential unit $\{u_t\}$  the unit $\{e^{-\lambda t}u_t\}$ is unital, and given a unital unit $\{u_t\}$ the unit  $\{e^{-\mu t}u_t\}$ is exponential. It is easily seen that this defines a bijection $\mathfrak{U}_\Omega(H)\to\mathfrak{U}(H)$. 
\end{rems}

\begin{defn} An addit for a spatial super product system $(H_t, U_{s,t})$, with canonical unit $\{\Omega_t\}$, is a measurable family of vectors $\{b_t:t\geq 0\}$ satisfying
\begin{itemize}
\item[(i)] $b_t\in H_t$ for all $t\geq 0$,
\item[(ii)] $U_{s,t}(b_s \otimes \Omega_t) +U_{s,t}(\Omega_s \otimes b_t) = b_{s+t}$ for all $s,t\geq0$.
\end{itemize}
We say an addit  is centred if $\ip{\Omega_t,b_t}=0$ for all $t\geq0$. Denote the set of all addits by $\mathfrak{A}(H)$. 
\end{defn}

Since $\Omega_t$ is a unit, every addit $b$ can be written as $b_t=c_t+\lambda_t \Omega_t$ such that $c$ is a centered addit and $\lambda_t \in \C$. 
Since $t \mapsto \lambda_t$ is measurable and $\lambda_{s+t}=\lambda_s+\lambda_t$, we have $\lambda_t=\lambda t$ for some $\lambda\in\C$.

\begin{lem}\label{additlemma} Let $b$ and $c$ be centered addits. Then
\begin{equation} \ip{b_t,c_t}=t\ip{b_1,c_1}. \label{additidentity}\end{equation}
\end{lem}

\begin{proof} The function $t \mapsto \ip{b_t,c_t}$ is measurable, and
$ \ip{b_{s+t},c_{s+t}}$ equals $$=\ip{U_{s,t}(b_s \otimes \Omega_t) +U_{s,t}(\Omega_s \otimes b_t) ,U_{s,t}(c_s \otimes \Omega_t) +U_{s,t}(\Omega_s \otimes c_t) },$$ but this is $\ip{b_s,c_s}+\ip{b_t,c_t}$, since the addits are centered.   \end{proof}

\section{Noncommutative It\^o Integrals}

In this section we develop It\^o Integrals on spatial super product systems, with respect to a centered addit. Using It\^o Integrals, we provide a bijection between centered addits and exponential units. Throughout this section we fix a spatial super product system $(H_t, U_{s,t})$ with the canonical unit $\{\Omega_t: t\geq 0\}$. 

\begin{defn}
An adapted process is a family $x=\{x_t:t\geq0\}$ satisfying $x_t\in H_t$ for all $t\geq 0$. 

An adapted process is simple if there exists a partition $0\leq s_0< s_1<\ldots $ of $\R_+$ into a countable family of intervals so that
$$x_t =U_{s_i, t-s_i}(x_i\otimes \Omega_{t-s_i})~~\mbox{if} ~~t \in [s_i, s_{i+1}),$$
where $x_i\in H_{s_i}$ for each $i\geq0$, and $\inf_{n\in\N} (s_{n}-s_{n-1})>0$.\end{defn}

\medskip

Let $x$ be a simple adapted process, $b$ a centred addit and $0\leq t_0\leq t_1$. Extend and redefine the partition for $x$ so that $t_0=s_m$, $t_1=s_n$ and define
$$\int_{t_0}^{t_1}{x_sdb_s}=\sum_{i=m}^{n-1}U_{s_i,s_{i+1}-s_i, t_1-s_{i+1}}(x_i\otimes b_{s_{i+1}-s_i}\otimes\Omega_{t_1-s_{i+1}}),$$ where $U_{r,s,t}:H_r\otimes H_s\otimes H_t\mapsto H_{r+s+t}$ is the canonical unitary operator well-defined by the associativity axiom (here for instance take $U_{r,s,t}= U_{r+ s , t}( U_{r ,
s} \otimes 1_{H_{t}})$) Clearly the new process $\{\int_0^t x_s db_s: t\geq 0\}$ is adapted. 

The definition of the above integral is well-defined. The fact that it does not depend on the partition with respect to which $x$ is simple,  follows from the additive property of the addit  $b$ and the multiplicative property of the unit $\{\Omega_t:t\geq 0\}$

Here is our version of It\^o's identity.

\begin{lem} Let $x,y$ be two simple adapted processes. Then
 $$\ip{\int_{t_0}^{t_1}{x_sdb_s},\int_{t_0}^{t_1}{y_sdb_s}}=\norm{b_1}^2\int_{t_0}^{t_1}{\ip{x_s,y_s}ds}.$$
\end{lem}

\begin{proof} We may assume that the partitions for $x$ and $y$ are the same. For  $i\neq j$ and $i<j$, note that
$$\ip{x_i\otimes b_{s_{i+1}-s_i}\otimes\Omega_{t_1-s_{i+1}}, y_j\otimes b_{s_{j+1}-s_j}\otimes\Omega_{t_1-s_{j+1}}}$$
$$= \ip{x_i \otimes b_{s_{i+1}-s_i} \otimes \Omega_{s_j -s_{i+1}},y_j}\ip{\Omega_{s_{j+1}-s_j}, b_{s_{j+1}-s_j}}=0.$$ Similarly, same is the case when $j<i$. 

Thus we have
\begin{eqnarray*}\ip{\int_{t_0}^{t_1}{x_sdb_s},\int_{t_0}^{t_1}{y_sdb_s}}&=&\sum_{i=m}^{n-1}{\ip{x_i\otimes b_{s_{i+1}-s_i},y_i \otimes b_{s_{i+1}-s_i}}}\\
&=&\sum_{i=m}^{n-1}{\ip{x_i,y_i}\norm{b_{s_{i+1}-s_i}}^2}\\
&=&\sum_{i=m}^{n-1}{\ip{x_i,y_i}({s_{i+1}}-{s_i})\norm{b_1}^2}\\
&=&\norm{b_1}^2\int_{t_0}^{t_1}{\ip{x_s,y_s}ds}.
\end{eqnarray*}   \end{proof}

\begin{defn} An adapted process $x$ is said to be a continuous process if the function $t \mapsto \|x_t\|^2$ is continuous. For a simple adapted process $x$ and centred addit $b$, the noncommutative It\^o integral $\int{x_sdb_s}$ of $x$ with respect to $b$ is the continuous process $t\mapsto \int_0^t{x_sdb_s}.$
\end{defn}

We say a sequence of adapted process $\{x^n:n \in \N\}$ converges to $x$ in $L^2$ on any finite interval $[a,b]$, if $\int_a^b \|x^n_t-x_t\|^2 dt$ converges to $0$.

\begin{prop}\label{propertiesprop} Let $x$ be a continuous adapted process such that there exists $F:\R_+\to\C$ analytic with $F(0)=0$ and, for all $s,t\geq0$, $\ip{x_{s+t},U_{s,t}(x_s\otimes \Omega_t)}=\norm{x_s}^2 F'(t)$. Then on any finite interval we can approximate $x$ in the $L^2$ norm by adapted step functions.
\end{prop}

\begin{proof} Pick an interval $I=[r,s]$ and set $r_0^n=r$, $r_{i+1}^n=r_i^n+(s-r)n^{-1}$ for $i=0,\ldots,n-1$. Then define a simple adapted process \begin{align*}x^n_t & = U_{r_i^n, t-r_i^n}(x_{r_i^n}\otimes \Omega_{t-r_i^n})~ \mbox{if}~x \in [r_i^n, r_{i+1}^n), \mbox{for}~ i=0,\ldots,n-1;\\
& = 0~\mbox{if} ~ t \geq s. \end{align*}
Then $\norm{x-x^n}^2_{L^2(I)}$ is
 \begin{align} &\sum_{i=0}^{n-1}\int_{r_i^n}^{r_{i+1}^n}{\norm{x_t-U_{r_i^n, t-r_i^n}(x_{r_i^n}\otimes \Omega_{t-r_i^n})}dt} \nonumber \\
  =&\sum_{i=0}^{n-1}\int_{r_i^n}^{r_{i+1}^n}{\norm{x_t}^2+\norm{x_{r_i^n}}^2-2Re\ip{x_t,U_{r_i^n, t-r_i^n}(x_{r_i^n}\otimes \Omega_{t-r_i^n})}dt} \nonumber \\
  =&\int_r^s{\norm{x_t}^2dt}-\sum_{i=0}^{n-1}{\norm{x_{r_i^n}}^2\int_{r_i^n}^{r_{i+1}^n}{2Re~F'(t-r_{i}^n)-1 dt}} \nonumber\\
  =&\int_r^s{\norm{x_t}^2dt}-\sum_{i=0}^{n-1}{\left(2Re~F\bigg(\frac{s-r}{n}\bigg)-\frac{s-r}{n}\right)\norm{x_{r_i^n}}^2}.
\label{sumequation}
 \end{align}
Since $t\mapsto \norm{x_t}^2$ is continuous and $F$ is analytic the sum in (\ref{sumequation}) tends to the Riemann integral $(2Re~F'(0)-1)\int_r^s{\norm{x_t}^2dt}$ as $n\to\infty$, and since $F'(0)=1$, we get the desired equality.   \end{proof}

Thus, if $x$ is a process satisfying the hypotheses of Proposition \ref{propertiesprop} we may define the It\^o integral of $x$ with respect to $b$ as follows. Take a sequence of adapted step functions $\{x^n\}$ converging to $x$  in the $L^2$ norm on $[t_0,t_1]$. Then, by It\^o's identity
$$\norm{\int_{t_0}^{t_1}{(x_s^n-x_s^m) db_s}}^2=\norm{b_1}^2\int_{t_0}^{t_1}{\norm{x_s^n-x_s^m}^2ds}\to 0$$
as $n,m\to \infty$, so the limit
$$\int_{t_0}^{t_1} x_s db_s:=\lim_{n\to\infty}\int_{t_0}^{t_1}{x_s^ndb_s}$$
exists. Moreover, if $y^n$ is another sequence of adapted step functions with $y^n\to x$ in $L^2$ norm on $[t_0,t_1]$, then
$$\norm{\int_{t_0}^{t_1}{(x_s^n-y_s^n)db_s}}^2=\norm{b_1}^2\int_{t_0}^{t_1}{\norm{x_s^n-y_s^n}^2ds}\to 0,$$
so the limit is independent of the chosen sequence of approximating functions. We call processes satisfying the hypotheses of Proposition \ref{propertiesprop} It\^o integrands.

The following general version of It\^o's lemma follows immediately from the very definition of It\^o integral.

\begin{lem} Let $x,y$ be two It\^o integrands, then
 $$\ip{\int_{t_0}^{t_1}{x_sdb_s},\int_{t_0}^{t_1}{y_sdb_s}}=\norm{b_1}^2\int_{t_0}^{t_1}{\ip{x_s,y_s}ds}.$$
\end{lem}

\medskip

For later use we record other properties of the It\^o integrals as a proposition below.

\begin{prop}\label{properties}
Let $x$ and $y$ be It\^o integrands.

\begin{itemize} 
 \item[(i)] $\ip{\int_{s}^{t} x_s dbs, \Omega_t}=0$ $\forall ~s \leq t$. 
 \medskip
 \item[(ii)] $\int^{s+t}_{t_0} x_s dbs=\int^{s}_{t_0} x_s dbs\otimes \Omega_t + \int^{s+t}_{s} x_s dbs$ $\forall ~t_0 \leq s,t$.
\medskip
\item[(iii)] $\int^{s+t}_s U_{s,r}(x_s \otimes y_r) db_r=U_{s,t}(x_s\otimes \int_0^t y_s db_s)$ $\forall ~s,t.$
\medskip
 \item[(iv)]  $\ip{\int_{s+t_0}^{s+t} x_s dbs, \int_{s_0}^{s} x_r db_r\otimes \Omega_t}=0$ $\forall s_0\leq s,t_0\leq t.$
\medskip
\item[(v)] $\int_s^{s+t} \Omega_t dbs=b_{s+t}-(b_s\otimes \Omega_t)$ $\forall~s,t.$ 
\end{itemize}
\end{prop}

\begin{proof}
All statements follow immediately, first by verifying for the simple adapted processes, and then by taking limits. For (iii) we need to use the associativity axiom of the super product system in addition.    \end{proof}

\medskip

\begin{rems} (i)
Notice that, if $b$ is an addit then $$\ip{b_{s+t},U_{s,t}(b_s\otimes \Omega_t)}=\ip{b_s\otimes \Omega_t+\Omega_s\otimes b_t,b_s\otimes \Omega_t}=\norm{b_s}^2,$$ so $b$ satisfies the hypotheses of Proposition \ref{propertiesprop} with $F(t)=t$. 

\noindent
(ii) If $x_t=\int_0^t{y_sdb_s}$ for some adapted process $y$, then $t\mapsto \|x_t\|$ is continuous and
$$\ip{x_{s+t},U_{s,t}(x_s\otimes \Omega_t)}=\norm{b_1}^2\int_0^{s}{\norm{y_r}^2}dr=\norm{x_s}^2.$$
Moreover $x_t$ is a limit of elements in $H_t$, so $x$ is an adapted process. $x$ satisfies the hypotheses of Proposition \ref{propertiesprop} with $F(t)=t$, hence is an It\^o integrand.
\end{rems}

\begin{prop}\label{QSDEprop} The non-commutative stochastic differential equation \begin{equation} u_t=\Omega_t+\int_0^{t}{u_sdb_s} \label{SDE}\end{equation} has a unique continuous solution. Moreover the solution is an exponential unit. \end{prop}

\begin{proof} Existence is given by Picard iteration. Set
 \begin{equation}x^1_t=b_t\quad\hbox{and}\quad x^{n+1}_t=\int_0^t{x^n_s db_s},\label{recursion} \end{equation}
for all $n\in\N$, $t\geq 0$. Note that (\ref{recursion}) well-defines $x^{n+1}$ since each $x^n$ satisfies the hypotheses of Proposition \ref{propertiesprop}. 
Then, thanks to It\^o's identity
$$\norm{\sum_{k=n}^m{x^k_t}}^2\leq\sum_{k=n}^m{\norm{b_1}^{2k}\frac{t^k}{k!}}$$
tends to $0$ as $m,n\to\infty$, so the series
$$u_t=\Omega_t+\sum_{n=1}^\infty{x_t^n}$$
converges. The measurability of the unit (in fact the continuity) follows from
\begin{align*} \norm{u_t -\left(u_s\otimes \Omega_{t-s}\right)}^2 & \leq \sum_{n=1}^\infty\norm{x^n_t-\left(x^n_s\otimes \Omega_{t-s}\right)}^2\\ &=\sum_{n=1}^\infty{\norm{b_1}^{2n}(t-s)^n/n!}\\ &=e^{\norm{b_1}^2(t-s)}-1 ~\forall ~0\leq s\leq t.\end{align*} Here we have used $(ii)$ of Proposition \ref{properties} and It\^o's identity.

For $m>0$, the iterated integrals satisfy
\begin{align*}\ip{x^n_s,x^{n+m}_t}&=\norm{b_1}^n\int_0^{s\wedge t}\int_0^{t_1}\cdots\int_0^{t_{n}}\ip{\Omega_{r_{n+1}},x^m_{r_{n+1}}}dr_{n+1}dr_{n}\cdots dr_1 \end{align*}
so we have the orthogonality relations $\ip{x^n_s,x^{n+m}_t}=0$. It follows that $u$ is an It\^o integrand with $F=t$, and a solution to \ref{SDE},  since
\begin{align*} \int_0^t{u_sdb_s} & =\lim_{n\to\infty}\left(\int_0^t{\Omega_t db_s+\sum_{k=1}^n{\int_0^tx^k_s}db_s}\right)\\ & =\lim_{n\to\infty}\left({b_t+\sum_{k=2}^n{x^{k+1}_t}}\right)=u_t-\Omega_t.\end{align*}

Now suppose that $u$ and $v$ are two continuous solutions of the QSDE. By iteration, we have
$$u_t-v_t=\int_0^t\int_0^{t_1}\cdots\int_0^{t_{n-1}}({u_{t_n}-v_{t_n})~db_{t_n}db_{t_{n-1}}\cdots db_{t_1}}$$
for each $n\in\N$. By continuity $M_t=\sup_{0\leq s\leq t}{\norm{u_s-v_s}}<\infty$ for each $t\geq0$ and thus
$$\norm{u_t-v_t}^2 \leq \norm{b_1}^{2n}\frac{M_t^2t^n}{n!}\to 0$$
as $n\to \infty$, thus $u=v$. 

Now we verify that $u_t$ is a unit for the super product system. Fix $s,t$ and define \begin{align*} v_r & =  u_r ~\mbox{if}~r \in (0,s)\\
& = U_{s,r}(u_s \otimes u_r) ~  \mbox{if}~r\geq s.
\end{align*} 
Then, using Proposition \ref{properties}, $ \Omega_{s+t}+\int_0^{s+t}v_r db_r$ is equal to
\begin{align*} &U_{s,t}(\Omega_s \otimes \Omega_t) + U_{s,t}(\int_0^s u_r db_r\otimes \Omega_t) +\int_s^{s+t} U_{s,r}(u_s \otimes u_r) db_r \\
=&U_{s,t}(u_s \otimes \Omega_t) +U_{s,t} (u_s \otimes \int_0^tu_r db_r) = U_{s,t}(u_s \otimes u_t)=v_{s+t}.\end{align*} Here we have used the fact that $u$ satisfies the stochastic differential equation. By the uniqueness of the solution to the equation \ref{SDE}, we get $U_{s,t}(u_s \otimes u_t)=u_{s+t}$. Hence $\{u_t\}$ is a unit.

Finally, since $\int_0^t{u_s^ndb_s}$ is orthogonal to $\Omega_t$ for each $t\geq0$, we have $\ip{u_t,\Omega_t}=1$ for all $t\geq0$, and $\{u_t\}$ is an exponential unit.    \end{proof}

For a centred addit $b$ define $\Exp_\Omega(b)$ to be the solution to the SDE (\ref{SDE}). Note that $$\norm{\Exp_\Omega(b)_t}^2=\ip{\Omega_t+\sum_{n=1}^\infty{x_t^n}, \Omega_t+\sum_{n=1}^\infty{x_t^n}}= e^{{\norm{b_1}^2t}}.$$ Moreover if $c$ is another centred addit satisfying  $$\Exp_\Omega(c)=u=\Exp_\Omega(d)$$ then 
$$0=\norm{\int_0^t{u_sd(b_s-c_s)}}^2=\norm{b_1-c_1}^2\int_0^t{\norm{u_s}^2ds},$$
so $b_1=c_1$ and hence $b=c$. Thus $\Exp_\Omega$ defines an injective map $\mathfrak{A}_\Omega(H)\to\mathfrak{U}_\Omega(H)$. In order to prove that $\Exp_\Omega$ is indeed a bijection, we explicitly provide the inverse in the following proposition.

We require the following observation on continuity. Every product system has a representation under which the units are semigroups of bounded operators (see \cite{arveson}) and since they are measurable the semigroups must be strongly continuous. It follows (for instance from the proof of Proposition \ref{HtOmega}) that the units of a product system are continuous in the following sense. Fix an arbitrary $T>0$ and for a unit $\{u_t:t\geq 0\}$ define \begin{align*}u'_t &=U_{t,T-t}(u_t\otimes \Omega_{T-t});\\ u'_{s,s+t}& = U_{s+t, T-(s+t)}(U_{s,t}(\Omega_s\otimes u_t)\otimes \Omega_{T-(s+t)}),\end{align*} for  $0\leq s,t\leq T$. Then the map $(s,t)\mapsto u'_{s,t}$ is continuous in $s,t$, for all $s,t\in [0,T]$. This also applies to super product systems, since the units of a super product system generate a product system.

\begin{prop}\label{log} For $t\geq 0$ and for a continuous  exponential unit $u$, set $$y^{i,n}_t= U_t^{2^n}(\Omega_{\frac{t}{2^n}}\otimes \cdots \otimes (u_{\frac{t}{2^n}}-\Omega_{\frac{t}{2^n}})\otimes \cdots \otimes \Omega_{\frac{t}{2^n}}),$$ with $(u_{\frac{t}{2^n}}-\Omega_{\frac{t}{2^n}}) $ at the $i-$th tensor and $U_t^{2^n}$ is the canonical unitary operator $H_{\frac{t}{2^n}}\otimes \cdots \otimes H_{\frac{t}{2^n}}\mapsto H_t$ well-defined by the associativity of the  super product system. Define $$y^n_t= \sum_{i=1}^{2^n} y^{i,n}_t,$$ then $\Log_\Omega(u)_t=\lim_{n\rightarrow \infty}y_t^n$ exists and $\{\Log_\Omega(u)_t: t \geq 0\}$ defines a centered addit.  
 \end{prop}

\begin{proof}
Set $\norm{u_s}^2=e^{\lambda s}$ for some $\lambda \in \R$. For any $s \geq 0$, we have $$\ip{u_s-\Omega_s, \Omega_s}=0; ~\|u_s-\Omega_s\|^2=e^{\lambda s}-1.$$ 
So we have $\|y_t^{i,n}\|^2=e^{\frac{\lambda t}{2^n}}-1$ and 
$\ip{y_t^{i,n},y_t^{i',n}}=0$ if $i \neq i'$.

For arbitrarily fixed  $m >n$ and $i=1,2,\ldots, 2^{n}$ let $$J_i = \{2^{m-n}(i-1)+1,2^{m-n}(i-1)+2, \ldots, 2^{m-n}i\}.$$ Let $J(i,j)$ denote the $j$-th element of $J_i$, then \begin{align*} \ip{y^{i, n}_t, y_t^{J(i',j), m}} &=0 \quad\mbox{if}\quad i\neq i';\\ \ip{y^{i, n}_t, y_t^{J(i,j), m}} &= e^{\frac{\lambda t}{2^m}}-1\quad (1\leq i\leq 2^n,~j\in J_i).\end{align*}

Now for $m>n$, after computations we find \begin{align*} \|y^n_t-y_t^m\|^2 &= 2^n (e^{\frac{\lambda t}{2^n}}-1) +2^m (e^{\frac{\lambda t}{2^m}}-1)-2\times2^n2^{m-n}(e^{\frac{\lambda t}{2^m}}-1)\\
&= 2^n(e^{\frac{\lambda t}{2^n}}-1)-2^m(e^{\frac{\lambda t}{2^m}}-1)\\
& = \sum_{k=2}^\infty \frac{(\lambda t)^k}{k!(2^n)^{k-1}}-\sum_{k=2}^\infty \frac{(\lambda t)^k}{k!(2^m)^{k-1}},\end{align*} which converges to $0$ as $n,m\rightarrow \infty$, hence $y_t^n$ converges, and we write $\Log_\Omega(u)_t$ for the limit.

Further \begin{align*}U_{t,t} ( y^n_t\otimes \Omega_t +\Omega_t\otimes y^n_t) &= U_{t,t}(\sum_{i=1}^{2^n} y^{i,n}_t \otimes \Omega_t+\Omega_t \otimes \sum_{i=1}^{2^n} y^{i,n}_t )\\ 
&=\sum_{j=1}^{2^n+1} y_{2t}^{j, n+1} 
 =y_{2t}^{{n+1}},\end{align*} where we have used the associativity axiom. This consequently implies that 
\begin{equation}\label{Utt} U_{t,t}(b_t\otimes \Omega_t +\Omega_t\otimes b_t)=b_{2t}, ~\forall ~t\geq 0.\end{equation}
Let us fix an arbitrary $T\geq 0$, and to make notation easier, embed $\{\Log_\Omega(u)_t:0\leq t\leq T\}$ into $H_T$. Denote \begin{align*}b'_t &=U_{t,T-t}(\Log_\Omega(u)_t\otimes \Omega_{T-t});\\ b'_{s,s+t}& = U_{s+t, T-(s+t)}(U_{s,t}(\Omega_s\otimes \Log_\Omega(u)_t)\otimes \Omega_{T-(s+t)}),\end{align*} for  $0\leq s,t\leq T$. Clearly to prove $\Log_\Omega(u)$ is an addit, we only need to verify that $b'_s+b'_{s,s+t}=b'_{s+t}$ for all $s,t\geq 0$.

Using (\ref{Utt}) and induction we have $$b'_{s,s+nt}=b'_{s,s+t}+b'_{s+t,s+2t}+ \cdots +b'_{s+(n-1)t,s+nt}~\forall s+nt\leq T.$$ Manipulations lead to \begin{equation}\label{brational}b'_{s,s+\frac{m}{n}t}=b'_{s,s+t}+b'_{s+t,s+\frac{n+1}{n}t}+ \cdots + b'_{s+\frac{m-1}{n}t,s+\frac{m}{n}t} \end{equation} for all $s+\frac{m}{n}t\leq T,~n\leq m$. Now let $s =\frac{p_1}{q_1}, t= \frac{p_2}{q_2}$ be rationals. Then \begin{align*} b'_{s+t}  &= b'_{\left(\frac{p_1q_2+p_2q_1}{p_1q_2}\right)s}\\
&= b'_s + b'_{s, \frac{p_1q_2+1}{p_1q_2}s}+\cdots +b'_{\frac{p_1q_2+p_2q_1-1}{p_1q_2}s, \frac{p_1q_2+p_2q_1}{p_1q_2}s}\\
&= b'_s + b'_{s,s+ \frac{1}{p_1q_2}s}+\cdots +b'_{s+\frac{p_2q_1-1}{p_1q_2}s,s+ \frac{p_2q_1}{p_1q_2}s}\\ &=b'_s+b'_{s,t}.
\end{align*} Here we have used relation \ref{brational} twice. 

Notice that $b'_{s,t}$ is the limit of $\sum_{i=1}^{2^n} \left(u'_{\frac{i-1}{2^n},\frac{i}{2^n}}-\Omega_T\right)$ and that the convergence is uniform in $[0,T]$. Now the continuity of $\{b'_{s,t}\}$ implies the relation $b'_{s+t}= b'_s+b'_{s,t}$  for all real $s,t$ such that $s+t\leq T$. Since $T$ is arbitrary we conclude that $\Log_\Omega(u)$ is indeed an addit.
Finally, since
$\ip{y_t^n,\Omega_t}=0$ for all $n \in \N,$ we have $\ip{\Log_\Omega(u)_t, \Omega_t}=0$ for all $t\geq 0$. Hence $\Log_\Omega(u)$ is a  centered addit.   \end{proof}

\medskip

\begin{thm} $\Log_\Omega$ and $\Exp_\Omega$ are mutually inverse maps between $\mathfrak{U}_\Omega(H)$ and $\mathfrak{A}_\Omega(H)$. \end{thm}

\begin{proof} Since $\Exp_\Omega:\mathfrak{A}_\Omega(H)\to\mathfrak{U}_\Omega(H)$ is injective we need only show that $\Exp_\Omega(\Log_\Omega(u))=u$ for all exponential units $u$. It suffices to show that
$$\norm{u_t}^2-2Re\ip{u_t,\Exp_\Omega(\Log_\Omega(u))_t}+\norm{\Exp_\Omega(\Log_\Omega(u))_t}^2=0.$$
Set $\norm{u_t}^2=e^{\lambda t}$, and let $y_t^{i,n}$ be as defined in proposition \ref{log}.
Then 
 $$\norm{\Log_\Omega(u)_t}^2= \lim_{n\rightarrow \infty}\sum_{i=1}^{2^n}\|y_t^{i,n}\|^2
=\lim_{n\rightarrow \infty} 2^n\left(e^{\frac{\lambda t}{2^n}}-1\right)=\lambda t,$$ and 
hence $\norm{\Exp_\Omega(\Log_\Omega(u))_t}^2=e^{\lambda t}.$ 
Similarly
$$\ip{u_t,\Log_\Omega(u)_t}=\lim_{n\rightarrow \infty}\sum_{i=1}^{2^n}\ip{u_t,y_t^{i,n}}
=\lim_{n\rightarrow \infty} 2^n\left(e^{\frac{\lambda t}{2^n}}-1\right)=\lambda t.$$
For any integrand $x$, $\ip{u_t,\int_0^tx_sd\Log_\Omega(u)_s}$ is equal to
\begin{align*}
 &\lim_{n\to\infty}\sum_{k=0}^{n-1}\ip{u_t, U_{\frac{kt}{n},\frac{t}{n}, t-\frac{(k+1)t}{n}}\left(x_{\frac{kt}{n}}\otimes \Log_\Omega(u)_{\frac{t}{n}}\otimes \Omega_{t-\frac{(k+1)t}{n}}\right)} \\
 =&\lim_{n\to\infty}\sum_{k=0}^{n-1}\ip{u_{kt/n},x_{kt/n}}\ip{u_{t/n},\Log_\Omega(u)_{t/n}}\\
 =&\lim_{n\to\infty}\sum_{k=0}^{n-1}\ip{u_{kt/n},x_{kt/n}}\lambda t/n =\lambda\int_0^t{\ip{u_s,x_s}ds}
\end{align*}
Thus in the notation of Proposition \ref{QSDEprop},
$$\ip{u_t,\Exp_\Omega(\Log_\Omega(u))_t}=1+\sum_{k=1}^\infty{\ip{u_t,x^k_t}}=\sum_{k=0}^\infty{\lambda^nt^n/n!}=e^{\lambda t}$$ as required.   \end{proof}

\begin{rem}\label{additsindex1} Units in a super product system gives rise to the covariance function, and we can associate an index, an isomorphic invariant, in precisely the same way as for product systems.  $\mathfrak{A}_\Omega(H)$ forms a Hilbert space with respect to the inner product $\ip{b,b'}=\ip{b_1,b'_1}$. The dimension of $\mathfrak{A}_\Omega(H)$ is equal to the index of the super product system, since we have $$\ip{\Exp_\Omega(b)_t, \Exp_\Omega(b')_t}=e^{\ip{b_1,b'_1}t}~\forall ~b,b' \in \mathfrak{A}_\Omega(H).$$
\end{rem}

\begin{rem}\label{additsindex} In particular the index of a (spatial) product system is equal to the dimension of the centered addits with respect to any fixed unit. In some examples of \en-semigroups on type I factors it is difficult to describe the units, but it may be easier to compute the addits. 

 This is the case for the CAR flow on $B(\Gamma_a(L^2((0,\infty),\k^\C)$ of rank $\dim(\k)$.  The CAR flow of rank $\dim(\k)$ is the \en-semigroup satisfying $$\theta_t(a(f)))=a(T_tf)\quad \forall f \in L^2((0,\infty),\k^\C),$$ where $T$ is the unilateral shift (see Example \ref{clifford}). It is well-known that CAR flow of rank $\dim(\k)$  is conjugate to the CCR flow of the same rank, hence completely spatial with index $\dim(\k)$.  Still, the units for the CAR flows are not known except the vacuum unit. 

 We can compute the addits for CAR flow easily. Setting $H=\Gamma_a(L^2((0,\infty),\k^\C)$, $H_t=\Gamma_a(L^2((0,t),\k^\C)$. Define unitary operators $U_t:H_t\otimes H\mapsto H$ by \begin{align*}U_t(\xi^t_1\wedge\xi^t_2\cdots \wedge\xi^t_n)\otimes (\xi_1\wedge\xi_2\cdots \wedge\xi_m))&\\=T_t\xi_1\wedge T_t\xi_2\cdots \wedge T_t\xi_m \wedge \xi^t_1\wedge\xi^t_2\cdots \wedge\xi^t_n.& \end{align*}  Then we have $\theta_t(X)=U_t\left( 1_{H_t}\otimes X\right)U_t^*.$  Hence the space of intertwiners  for $\gamma_t$ is given by $\{T_{\xi_t}:\xi_t\in H_t \}$ with \begin{equation}\label{T}T_{\xi_t}(\xi)=U_t(\xi_t\otimes \xi),~\forall \xi \in H.\end{equation}  The product system of $\gamma_t$ is isomorphic to $(H_t, U_{s,t})$ with \begin{align*} U_{s,t}(\xi^s_1\wedge\xi^s_2\cdots \wedge\xi^s_n)\otimes (\xi^t_1\wedge\xi^t_2\cdots \wedge\xi^t_m))&\\=T_s\xi^t_1\wedge T_s\xi^t_2\cdots \wedge T_s\xi^t_m \wedge \xi^s_1\wedge\xi^s_2\cdots \wedge\xi^s_n.&\end{align*}  The space of centered addits $\mathfrak{A}_\Omega(H)$ is given by $\{x1_{0,t}: x \in \k^\C\}$ (see Corollary \ref{cliffordaddits}). This proves that the index of $\gamma_t$  is $\dim(\k)$. 
\end{rem}


\section{The noncommutative white noise}

In this section we show that every \en-semigroup $\{\alpha_t; t \geq 0\}$ on a \twoone factor $\m$ has an associated noncommutative white noise. We define unital and additive cocycles for the noise and collect some facts which will be useful in the sequel.

\begin{defn} Let $\mathcal{A}_1,\mathcal{A}_2\subseteq \mathrm{B}$ be *-algebras and suppose $\varphi$ is a positive linear functional on $\mathrm{B}$. $\mathcal{A}_1$ and $\mathcal{A}_2$ are said to be $\varphi$-independent if 
$$\varphi(xy)=\varphi(x)\varphi(y) \quad\hbox{for~all}\quad x\in\mathcal{A}_1,~y\in\mathcal{A}_2.$$
\end{defn}

\begin{prop} Let $U$ be a gauge cocycle for $\alpha$. Then $U_t$ is independent of $\alpha_t(\m)$ in the sense that
 $$\tau(U_t\alpha_t(x))=\tau(U_t)\tau(\alpha_t(x))\quad\hbox{for~all}\quad x\in\m,~t\geq0.$$
\end{prop}

\begin{proof} We simply compute
 \begin{eqnarray*} \tau(U_t\alpha_t(x))&=&\ip{\Omega,U_t\alpha_t(x)\Omega}=\ip{\Omega,U_tS_tx\Omega}\\
  &=& \ip{\Omega,S_t^*U_tS_tx\Omega}=\ip{\Omega,e^{tc_*(1,U)}x\Omega}\\
&=&e^{tc_*(1,U)}\tau(x)=\tau(U_t)\tau(\alpha_t(x))
 \end{eqnarray*}
for all $t\geq0$, $x\in\m$.   \end{proof}

This property will be called the future independence property of gauge cocycles. Let $\mathcal{A}_{t}$ be the *-algebra generated by elements
 $$\{U_s:~U\in G(\alpha),~0\leq s\leq t\}.$$
A simple induction argument shows that $\mathcal{A}_t$ is the linear span of elements of the form
$$U^1_{s_1}\alpha_{s_1}(U^2_{s_2}\alpha_{s_2}(\cdots \alpha_{s_{n-1}}(U^n_{s_n})\cdots))$$
where $n\in\N$, $U^1,\ldots U^n\in G(\alpha)$, $s_1,\ldots,s_n\in \R_+$ and $s_1+\ldots+s_n\leq t$. Let $\alg_{0,t}$ be the ultraweak closure of $\mathcal{A}_t$ and set $\alg_{s,s+t}:=\alpha_s(\alg_{0,t})$. 

\begin{prop}  $(\alg_{s,t})_{0\leq s\leq t}$ is a covariant filtration of von Neumann algebras for $\alpha$. Moreover, when $t_1\leq s_2$ the algebras $\alg_{s_1,t_1}$ and $\alg_{s_2,t_2}$ are independent.
\end{prop}

\begin{proof} The covariance easily follows as
$$\alpha_r(\alg_{s,t})=\alpha_r\circ\alpha_{s}(\alg_{0,t-s})=\alg_{r+s,r+t}.$$
We need to show that if $[s_1,t_1]\subseteq [s_2,t_2]$ then $\alg_{s_1,t_1}\subseteq \alg_{s_2,t_2}$, i.e. that $\alpha_{s_1-s_2}\alg_{0,t_1-s_1}\subseteq \alg_{0,t_2-s_2}$. For this we note that when $0\leq p\leq t_1-s_1$ we have $\alpha_{s_1-s_2}(U_{p})=U_{s_1-s_2}^*U_{p+s_1-s_2}$ and this is in $\alg_{0,t_2-s_2}$ as $0\leq s_1-s_2\leq t_2-s_2$ and $0\leq p+s_1-s_2\leq t_1-s_2\leq t_2-s_2$. Since these elements generate $\alpha_{s_1-s_2}(\alg_{0,t_1-s_1})$ the inclusion follows.

Now let $t_1\leq s_2$. By ultraweak continuity of $\tau$, the final part of the theorem follows if $\alpha_{s_1}(\mathcal{A}_{t_1-s_1})$ and $\alg_{s_2,t_2}$ are independent. Hence it suffices to show that if $x\in\factor$, $n\in\N$, $U^1,\ldots,U^n\in G(\alpha)$ and $u_1,\ldots,u_n\in\R_+$ with $u_1+\ldots+u_n\leq t_1-s_1$, the elements $$\alpha_{s_1}(U^1_{u_1}\alpha_{u_1}(U^2_{u_2}\alpha_{u_2}(\cdots \alpha_{u_{n-1}}(U^n_{u_n})\cdots)))\quad\text{and}\quad\alpha_{s_2}(x)$$ are independent.
Since $u_1\leq t_1-s_1\leq s_2-s_1$ the future independence property of gauge cocycles gives us
\begin{align*}
&\tau(\alpha_{s_1}(U^1_{u_1}\alpha_{u_1}(U^2_{u_2}\alpha_{u_2}(\cdots \alpha_{u_{n-1}}(U^{n}_{u_{n}})\cdots))) \alpha_{s_2}(x))\\
&=\tau(U^1_{u_1}\alpha_{u_1}(U^2_{u_2}\alpha_{u_2}(\cdots \alpha_{u_{n-1}}(U^{n}_{u_{n}})\cdots)) \alpha_{s_2-s_1}(x))\\
&=\tau(U^1_{u_1})\tau(\alpha_{u_1}(U^2_{u_2}\alpha_{u_2}(\cdots \alpha_{u_{n-1}}(U^{n}_{u_{n}})\cdots)) \alpha_{s_2-s_1}(x)).
\end{align*}
Moreover, as $u_k\leq s_2-s_1-\sum_{j=1}^{k-1}{u_j}$ for each $1\leq k\leq n$, we can continue inductively to obtain
\begin{align*}
&\tau(\alpha_{s_1}(U^1_{u_1}\alpha_{u_1}(U^2_{u_2}\alpha_{u_2}(\cdots \alpha_{u_{n-1}}(U^{n}_{u_{n}})\cdots))) \alpha_{s_2}(x))\\
&=\tau(U^1_{u_1})\tau(U^2_{u_2})\cdots \tau(U^n_{u_n})\tau(\alpha_{s_2-s_1-\sum_{j=1}^{n-1}{u_j}}(x)).
\end{align*}
Finally, we may use the future independence property and invariance of $\tau$ under $\alpha$ to note that
$\tau(U^1_{u_1})\tau(U^2_{u_2})\cdots \tau(U^n_{u_n})$ $$=\tau(\alpha_{s_1}(U^1_{u_1}\alpha_{u_1}(U^2_{u_2}\alpha_{u_2}(\cdots \alpha_{u_{n-1}}(U^{n}_{u_{n}})\cdots))))$$
and $\tau(\alpha_{s_2-s_1-\sum_{j=1}^{n-1}{u_j}}(x))=\tau(\alpha_{s_2}(x))$.   \end{proof}

\begin{rem} In practice, when we define an \en-semigroup, other, quite natural filtrations present themselves. By design, the collection of cocycles adapted to $(\alg_{s,t})_{0\leq s\leq t}$ is the gauge group $G(\alpha)$. On the other hand it is not clear if the collection of cocycles adapted to another filtration carries any meaningful information about the cocycle conjugacy class of $\alpha$. \end{rem}

\begin{defn}\label{whitenoisedef} Set $\alg=\bigvee_{0\leq t} \alg_{0,t}$, $\varphi=\tau|_\alg$ and $\sigma=\alpha|_\alg$, then we call the quintuple $(\alg,\varphi,\sigma,(\alg_{s,t})_{0\leq s\leq t})$ the noncommutative white noise associated to the \en-semigroup $\alpha$.

Define $G^\alpha_{t}$ to be the subspace of $H$ generated by $\alg_{0,t}$.  
We denote the projection onto $G^\alpha_t$ by $P_t$. Finally, let $H^t:=S_tH$, the subspace generated by $\alpha_t(\factor)$. \end{defn}

\begin{rem} Similar forms of noncommutative white noise have been well studied - see the review article \cite{kostler}, or section 6.5 of \cite{bernoulli}. The key difference being the use of a \emph{group} of \emph{automorphisms}. This gives good continuity properties for the family of conditional expectations $\expectation_{s,t}:\alg\to\alg_{s,t}$ (see Lemma 3.1.5 in \cite{bernoulli}), allowing the development of noncommutative It\^o integrals.
\end{rem}

\begin{lem}\label{multlemma} The multiplication $\alg_{0,t}\times H^t\to H$, $(x,\xi)\mapsto x\xi$ extends to a continuous bilinear map
$$G^\alpha_t\times H^t\to H, \qquad (\xi,\eta)\mapsto \xi\eta.$$ 
Moreover for $\xi_1,\xi_2\in G^\alpha_t$ and $\eta_1,\eta_2\in H^t$ we have
$$\ip{\xi_1\eta_1,\xi_2\eta_2}=\ip{\xi_1,\xi_2}\ip{\eta_1,\eta_2}.$$
\end{lem}

\begin{proof} We first show that, for any $\eta\in H^t$ and $x\in\alg_{0,t}$, we have
$$\norm{x\eta}_H^2=\tau(x^*x)\norm{\eta}_H^2.$$
indeed, pick $y_i\in\alpha_t(\factor)$ with $y_i\to \eta$. Then as $x$ is an operator on $H$, we have
$\norm{x\eta-xy_i}_H\leq \norm{x}\norm{\eta-y_i}_H\to0$, so $xy_i\to x\eta$. Hence
$$\norm{x\eta}_H^2=\lim_{i\to\infty}\tau(x^*y_i^*y_ix)=\tau(x^*x)\norm{\eta}_H^2$$
by independence.

Now pick an arbitrary $\xi\in G^\alpha_t$ and $\eta\in H^t$. Since $G^\alpha_t$ is the norm closure of $\alg_{0,t}$ in $H$ there is a sequence $x_i\in\alg_{0,t}$ tending to $\xi$. The sequence $x_i\eta$ is Cauchy, since
$$\norm{(x_i-x_j)\eta}_H^2=\tau((x_i-x_j)^*(x_i-x_j))\norm{\xi}_H^2\to 0.$$
We call the limit $\xi\eta$. By continuity of addition and scalar multiplication the map $(\xi,\eta)\mapsto \xi\eta$ is bilinear and by definition it extends the usual multiplication on $\alg_{0,t}\times H^t$. The norm of the element $\xi\eta$ is given by
$$\norm{\xi\eta}_H^2=\lim_{i\to\infty}\norm{x_i\eta}_H^2=\lim_{i\to\infty}\tau(x_i^*x_i)\norm{\eta}_H^2=\norm{\xi}_H^2\norm{\eta}_H^2,$$
hence the multiplication is continuous.

If $x_1,x_2\in \alg_{0,t}$ and $y_1,y_2\in\alpha_t(\factor)$ then we have
$$\ip{x_1y_1,x_2y_2}_H=\tau(y_1^*x_1^*x_2y_2)=\tau(x_1^*x_2)\tau(y_1^*y_2)=\ip{x_1,x_2}_H\ip{y_1,y_2}_H.$$
The final claim of the lemma now follows from standard limiting arguments.   \end{proof}

As a consequence to the above lemma we have, for $\xi\in G^\alpha_t$ and $\eta\in H^t$,
$\ip{\xi,\eta}=\ip{\xi,\Omega}\ip{\Omega,\eta}$. The following corollary follows immediately.

\begin{cor}
$(G^\alpha_t, U_{s,t})$ forms a completely spatial product system with $$U_{s,t}(\xi_s \otimes \xi_t)= \xi_s S_s\xi_t,~~ \forall ~ \xi_s \in G^\alpha_s, \xi_t \in G^\alpha_t,$$ and $\Omega_t=\Omega$ for all $t \geq 0$.  
\end{cor}

\medskip

The units in the product system $(G^\alpha_t, U_{s,t})$ are multiplicative cocycles, that is $\{u_t:u_t \in G^\alpha_t\} \subseteq H$ satisfying $u_sS_su_t=u_{s+t}$. Similarly addits are additive cocycles $\{b_t:b_t \in G^\alpha_t\} \subseteq H$ satisfying $b_s+S_sb_t=b_{s+t}.$ We denote the unital cocycles and the  exponential cocycles by $\mathfrak{U}(\alpha)$ and $\mathfrak{U}_\Omega(\alpha)$ respectively. Also we denote by $\mathfrak{A}(\alpha)$ the space of all additive cocycles and denote by $\mathfrak{A}_0(\alpha)$ the subset of $\mathfrak{A}(\alpha)$ satisfying the structure equation
\begin{equation}\label{structureeqn} \ip{b_t-P_0b_t,b_t-P_0b_t}+P_0b_t+P_0Jb_t=0. \end{equation}
where $P_0$ is the projection onto $\C\Omega_t$. 
Let $c_t=b_t+\lambda t\Omega$, where $b$ is a centred additive cocycle. Using Lemma \ref{additlemma}, the structure equation (\ref{structureeqn}) can be rewritten as
\begin{equation}\label{structureeqn2} \norm{b_1}^2+\lambda+\overline{\lambda}=0.\end{equation}

Let $c_t=b_t+\lambda t\Omega$, where $b\in\mathfrak{A}_\Omega$. Then we define $\Exp(c)_t:=e^{\lambda t}\Exp_\Omega(b)_t$ for all $t\geq0$. Since
$$\norm{\Exp(c)_t}^2=e^{(\norm{b_1}^2+\lambda+\overline{\lambda})t}$$
we see that $\Exp(c)$ is a unital cocycle precisely when $c$ satisfies the structure equation (\ref{structureeqn2}). Conversely, given a unital cocycle $u$ there exists $\lambda\in\C$ such that $\ip{\Omega,u_t}=e^{\lambda t}$. Since $v_t=e^{-\lambda t}u_t$ is an exponential cocycle the assignment $\Log(u)_t=\Log_\Omega(v)_t+\lambda t\Omega$ gives an inverse for $\Exp:\mathfrak{A}_0(\alpha)\to\mathfrak{U}(\alpha)$.

\begin{cor}\label{isometrycor} There exists an isometry $D$ from $H(G(\alpha))$ into the subspace
 $$\overline{\lin}\{b_1:~b\in\mathfrak{A}_0~\text{centred}\}.$$
\end{cor}

\begin{proof} If $U$ and $V$ are gauge cocycles, and $u$, $v$ their corresponding unital cocycles there exist centred addits $b$ and $c$ and constants $\lambda,\mu\in\C$ such that $\Log(u)_t=b_t+\lambda t\Omega$ and $\Log(v)_t=c_t+\mu t\Omega$. Moreover
$$\ip{u_t,v_t}=e^{(\overline{\lambda}+\mu+\ip{b_1,c_1})t},$$
i.e. $c_*(U,V)=\overline{\lambda}+\mu+\ip{b_1,c_1}$. The result now follows from \cite{arveson}, Proposition 2.6.9.   \end{proof}

\begin{cor}\label{trivcor} The gauge index $\dim G(\alpha)$ is zero if and only if the gauge group is trivial, i.e. $G(\alpha)\cong\mathbb{R}$. \end{cor}

\begin{proof} Recall that $H(G(\alpha))$ is the completion of a quotient space of functions and denote the equivalence class of a function $f$ by $[f]$. The proof of Corollary \ref{isometrycor} shows that
$$D(\sum_{U\in G(\alpha)}{a_U [\delta_{U}]})=\sum_{U\in G(\alpha)}{a_U(1-P_0)\Log(U\Omega)_1},$$ 
whenever finitely many of the $a_U\in\C$ are nonzero and we have that $\sum_{U\in G(\alpha)}{a_U}=0.$

If the gauge group is trivial every gauge cocycle is of the form $U_t=e^{i\theta t} 1$ for some $\theta\in\R$. Thus $\Log(U\Omega)_t\in P_0H$, so $D(\sum_{G(\alpha)}{a_U[\delta_{U}]})=0$, hence $H(G(\alpha))\cong\{0\}$. Conversely, let $\dim G(\alpha)=0$ and pick $U\in G(\alpha)$. We must have $D([\delta_U-\delta_1])=(1-P_0)\Log(U\Omega)_1=0$, so there exists $\lambda\in \C$ with $\Log(U\Omega)_t=\lambda t\Omega$. Writing $U_t=e^{\lambda t}1$ we see that $\lambda=i\theta$ for some $\theta\in\R$.   \end{proof}

\begin{rem} There is a natural action of $G(\alpha)$ as automorphisms of the product system $(G^\alpha_t,U_{s,t})$ via $(U\cdot \xi)_t=U_t\xi_t$ for any measurable section $\xi$. It follows that the gauge group of an \en-semigroup on a \twoone factor is always a subgroup of one of the groups $G_H$ in \cite{arveson}, Section 3.8.
\end{rem}

\section{Computation of Gauge index}\label{gauge}

Let $\alpha$ be an \en-semigroup on a II$_1$ factor $\m\subseteq B(L^2(\m))=H$ and $S=\{S_t:t\geq 0\}$ be the canonical unit for $\alpha$.  

\begin{defn} An additive cocycle for $S$ is a continuous map $b:\R_+\to H$ satisfying $b_s+S_sb_t=b_{s+t}$ for all $s,t\geq 0$,
for any $\xi\in H$. 
\end{defn}

\begin{rem}\label{Sadditsconjugate} If we define a centred additive cocycle for the unit $T$ as a continuous family $(b_t)\in H$ satisfying $b_s+T_sb_t=b_{s+t}$ and $\ip{b_t,T_t\xi}=0$ for all $s,t\geq0$, $\xi\in H$ then, as in the proof of Lemma \ref{additlemma},  centred additive cocycles satisfy $\ip{b_s,c_s}=s\ip{b_1,c_1}$ for all $s\geq0$.
Hence additive cocycles form a Hilbert space via the inner product $\ip{b,c}:=\ip{b_1,c_1}_H$. It is easily seen that if $\alpha$ and $\beta$ are conjugate \en-semigroups then the spaces of centred additive cocycles for their respective canonical units are isomorphic. But they need not be invariant under cocycle conjugacy. 
\end{rem}

\begin{lem}\label{fockaddits} Let $\gamma^n$ be the free flow of index $n$ on the II$_1$ factor $L(F_\infty)\subseteq B(\Gamma_f((L^2(0,\infty), \k^\C))$. Every additive cocycle $\{b_t: t \geq 0\}$ for $S$, satisfying an additional condition $b_t \in \Gamma_f(L^2((0,t),\k^\C)$ is of the form $b_t=\lambda t+v1_{[0,t]}$ where $\lambda\in\C$, $v\in\k^\C$. \end{lem}

\begin{proof} Note that $$S_t(f_1\otimes\cdots \otimes f_n)=T_tf_1\otimes\cdots\otimes T_tf_n.$$ 
As $S$ leaves each of the spaces $L^2(\R_+;\k^\C)\opower{n}$ invariant each addit decomposes as $b=\sum^\oplus_{n\geq 0} b^n$, where $b^n$ is an addit in $L^2(\R_+;\k^\C)\opower{n}$. Since $b^0_{s+t}=b^0_s+S_sb^0_t=b^0_s+b^0_t$, we have $b^0_t=\lambda t$ for some $\lambda\in\C$. 

Pick an orthonormal basis $\{e_i:~i\in I\}$ for $\k$. Since $b^1_s\in L^2([0,s];\k^\C)$ and since $r\mapsto 1_{[0,r]}\otimes e_i$ is a centred addit we have for all $0\leq r\leq s$ and $i\in I$ 
\begin{eqnarray*} \int_0^r{\ip{e_i,b^1_s(t)}dt}&=&\ip{1_{[0,r]}\otimes e_i,b^1_s}=\ip{1_{[0,r]}\otimes e_i,b^1_r+S_rb^1_{s-r}}\\
&=&\ip{1_{[0,r]}\otimes e_i,b^1_r}=\mu_i r\end{eqnarray*}
for some $\mu_i\in\C$.
Thus $b^1_s=1_{[0,s]}\otimes v$, where $v=\sum {\mu_ie_i}\in\k^\C$.

Fix $n\geq 2$ and set $F_s:=b^n_s\in L^2(\R_+;\k^\C)\opower{n}\simeq L^2(\R_+^n;(\k^\C)\opower{n})$ for each  $s\geq0$. Then by induction we have, for any $k\in\N$,
$$F_s=\sum_{i=0}^{2^k-1} S_{2^{-k}is}F_{2^{-k}s}.$$
Thus, modulo a null set,	
$$\supp F_s\subset \bigcup_{i=0}^{2^k-1}[2^{-k}is,2^{-k}(i+1)s]^{\times n}$$
and since this holds for every $k\geq 0$ we get $$\supp F_s\subset \{x\in\R_+^n:~x_1=x_2=\ldots=x_n\},$$ which has measure zero.   \end{proof}

For Clifford flow on the hyperfinite II$_1$ factor $\mathcal{R}$ the canonical unit is the second quantized shift on the antisymmetric Fock space $\Gamma_a(L^2((0,\infty),\k))$. But this is the restriction of the second quantized shift on free Fock space to subspace generated by the antisymmetric tensors. So the following corollary to lemma \ref{fockaddits} follows immediately.

\begin{cor}\label{cliffordaddits}
Let $\alpha^n$ be the Clifford flow of index $n$ on the hyperfinite II$_1$ factor $\mathcal{R}\subseteq B(\Gamma_a((L^2(0,\infty), \k^\C))$. Every additive cocycle $\{b_t:t \geq 0\}$ for the canonical unit $S$, satisfying $b_t \in \Gamma_a(L^2((0,t),\k^\C))$, is of the form $b_t=\lambda t+v1_{[0,t]}$ where $\lambda\in\C$, $v\in\k^\C$. 
\end{cor}

\begin{lem}\label{N} Let $\n$ be a von Neumann subalgebra of a II$_1$ factor $\m$ and $\Omega$ a cyclic and separating trace vector for $\m$. If $\xi=x\Omega$ for some $x\in \m$, $x\notin\n$ then $\xi\notin \overline{\n\Omega}$. \end{lem}

\begin{proof} Suppose we have a sequence of vectors $x_n\Omega$ with $x_n\in\n$ for all $n$ and $x_n\Omega\to\xi$. Then, using the cyclic property of the trace, $x_n^*\Omega\to x^*\Omega$ and, for all $y\in \m'$, $z\in \n'$,
\begin{align*} \ip{y\Omega,zx\Omega}&=\lim_{n\to\infty}{\ip{y\Omega,zx_n\Omega}}=\lim_{n\to\infty}{\ip{yx_n^*,z\Omega}}\\
 &=\ip{yx^*\Omega,z\Omega}=\ip{y\Omega,xz\Omega}.
\end{align*}
Using density of $\m'\Omega$ and faithfulness of the trace this implies $x\in\n$, a contradiction.   \end{proof}

\begin{prop} If $\alpha^n$ is the Clifford flow of rank $n$ then the gauge index $\dim G(\alpha^n)=0$. \end{prop}

\begin{proof} By  \cite{alev}, Proposition 2.9, the relative commutant $\mathcal{R}_t=\alpha_t(\mathcal{R})'\cap\mathcal{R}$ is the von Neumann algebra generated by even polynomials in the elements $\{u(f):~\supp f \subset [0,t]\}$. Note that $\mathcal{R}_t\Omega= G^\alpha_t \subseteq \Gamma_a(L^2((0,t),\k^\C))$, and that any addit for $\{G^\alpha_t:t\geq 0\}$ is an additive cocycle for $\{S_t:t \geq 0\}$. Now it follows from corollary \ref{cliffordaddits} and lemma \ref{N} that the only centred addit is thus the zero addit. Thanks to \ref{isometrycor} we have
$$\dim G(\alpha)\leq \dim \overline{\lin}\left\{b_1:~b\in\mathfrak{A}_0~\text{centred}\right\}=0.$$   \end{proof}

Thus it follows from Corollary \ref{trivcor} that for Clifford flows on $\hyperfinite$ we have $G(\alpha)\cong \R$.

\begin{prop}
If $\beta^n$ is the even Clifford flow of rank $n$ then the gauge index $\dim G(\beta^n)=0$. \end{prop}

\begin{proof}
The GNS Hilbert space for even Clifford flows of rank $n$ is the subspace of the antisymmetric Fock space generated by the even particle spaces, that is $$H^e:=[\xi_1 \wedge\xi_2 \wedge \cdots \wedge \xi_{2m};~ \xi_1, \xi_2 \cdots \xi_{2m} \in L^2(0,\infty), \k^\C), ~m \in \N_0],$$ where $\k$ is a Hilbert space with dimension $n$.  Again the canonical unit for this even Clifford flow is the restriction of the second quantized shift to this space. 
Also the relative commutant $\mathcal{R}_t=\alpha_t(\mathcal{R})'\cap\mathcal{R}$ is again the von Neumann algebra generated by even polynomials in the elements  $\{u(f):~\supp f\subset [0,t]\}$, as in the case of Clifford flows, and  $ G^\alpha_t \subseteq  H^e_t$, where $$H^e_t:=[\xi_1 \wedge\xi_2 \wedge \cdots \wedge \xi_{2m};~ \xi_1, \xi_2 \cdots \xi_{2m} \in L^2(0,t), \k^\C), ~m \in \N_0].$$
By our earlier analysis the canonical unit does not admit any non-trivial additive cocycle (apart from $\{\lambda t\Omega\}$) in this space.    \end{proof}

The following lemma (and its proof) was pointed out to the first named author by Jesse Peterson on the discussion site mathoverflow.net.

\begin{lem}
 Let $\m_1$ and $\m_2$ be \twoone factors acting standardly with traces $\tau_1$ and $\tau_2$ and set $\m=\m_1*\m_2$, $\tau=\tau_1*\tau_2$. Then the relative commutant $\m_1'\cap\m$ in $L^2(\m,\tau)$ is trivial.
\end{lem}

\begin{proof}
Recall that $L^2(\m,\tau)$ decomposes as \begin{equation}\label{freedecomposition}  \C\oplus \bigoplus_{n\in\N}\bigoplus_{i_1\neq i_2\neq \ldots i_n} L^2_0(\m_{i_1},\tau_{i_1})\otimes\cdots\otimes L^2_0(\m_{i_n},\tau_{i_n})\end{equation} where $L^2_0(\m_{i_j},\tau_{i_j})$ is the orthogonal complement of the identity in the respective $L^2$ space (see \cite{voiculescu}). $L^2(\m,\tau)$ is naturally an $\m_1$-bimodule and under this decomposition the left action of $\m_1$ is given by
$$x\cdot \xi_1\otimes\cdots\otimes\xi_n=\left\{\begin{array}{ll} x\xi_1\otimes\cdots\otimes\xi_n & \text{if }\xi_1\in L^2(\m_1,\tau_1), \\ x\otimes\xi_1\otimes\cdots\otimes\xi_n & \text{if } \xi_1\in L^2_0(\m_2,\tau_2), \end{array} \right.$$
with the right action being defined similarly.

Picking any vector $\xi=\xi_1\otimes\cdots\otimes \xi_n$ beginning and ending with nontrivial elements of $L^2_0(\m_2,\tau_2)$ we see that it generates an $\m_1$-sub-bimodule $L^2(\m_1,\tau_1)\otimes \xi\otimes L^2(\m_1,\tau_1)$. From the decomposition (\ref{freedecomposition}) it follows easily that, as an $\m_1$-bimodule, $L^2(\m,\tau)$ decomposes into a direct sum of the standard bimodule $L^2(\m_1,\tau_1)$ and a countable number of these. Moreover each $L^2(\m_1,\tau_1)\otimes \xi\otimes L^2(\m_1,\tau_1)$ is canonically isomorphic to the bimodule $L^2(\m_1,\tau_1)\otimes L^2(\m_1,\tau_1)$ with action $x\cdot (a\otimes b)\cdot y=xa\otimes by$ (the coarse bimodule). 

Any element of $\m_1'\cap (\m_1*\m_2)$ corresponds to a vector $\xi\in L^2(\m,\tau)$ satisfying $x\cdot\xi=\xi\cdot x$. Since $\m_1$ is a factor we are left to characterise the central vectors in the coarse bimodules. Endowing the space of Hilbert-Schmidt operators $HS(L^2(\m_1,\tau_1))$ with the standard $\m_1$-bimodule structure coming from left and right multiplication one gets an isomorphism $L^2(\m_1,\tau_1)\otimes L^2(\m_1,\tau_1)\cong HS(L^2(\m_1,\tau_1))$. Thus a central vector in the coarse bimodule corresponds to a Hilbert-Schmidt operator in $B(L^2(\m_1,\tau_1))$ which lies in the commutant $\m_1'$ of $\m_1$ in $L^2(\m_1,\tau_1)$. For any such operator, upon taking a spectral projection, we get a finite dimensional projection living in $\m_1'$. But this is a \twoone factor, hence the coarse bimodule contains no central vectors.   \end{proof}

The following Corollary to the above lemma also follows from Proposition \ref{freesuper} in Section \ref{coupling}.

\begin{prop} Let $\gamma^n$ be the free flow on $\Phi(\k)$ of rank $n$. Then we have $\dim G(\gamma^n)=0$. \end{prop}

\begin{proof} 
If $\Phi_t(\k) = \{s(f_t): f_t \in L^2(0,t)\}''$ and $\Phi_{[t}(\k) = \{s(f_{[t}): f_{[t} \in L^2(t,\infty)\}'',$ then $\Phi(\k)=\Phi_t(\k)*\Phi_{[t}(\k)$. It follows that $\left(\alpha_t(\Phi(\k))\right)'\cap \Phi(\k)=\C$.    \end{proof}

Triviality of the gauge group implies the following rigidity condition for the semigroup. A history $(\gamma,\m)$ on $B(H)$ consists of a group of automorphisms on $B(H)$ and an invariant von Neumann subalgebra $\m\subset B(H)$ (see \cite{interactions} or \cite{arveson}). Any history induces a pair of \en-semigroups via $\alpha_t:=\gamma_t|_{\m}$ and $\beta_t:=\gamma_{-t}|_{\m'}$ ($t\geq0$). Conversely, given a pair of \en-semigroups $\alpha$, $\beta$ on $\m$ and $\m'$ respectively the question of whether they extend to a history on $B(H)$ is still unsolved. Arveson showed that when they do exist, the number of distinct histories inducing the pair is exactly parametrised by the gauge group of $\alpha$ (\cite{arveson}, Proposition 2.8.4). Hence if $\alpha$ is a Clifford, even Clifford, or free flow on $\m$ and $\beta$ is any \en-semigroup on $\m'$, then a history extending $\alpha$ and $\beta$ must be unique. Physically, it is completely determined by its ``past'' and ``future'' dynamics. Whether this rigidity problem always holds, or if there exist examples of \en-semigroups on \twoone factors with nontrivial gauge groups, is an interesting open question.
 

\section{Computation of coupling index}\label{coupling}
 
Let $\alpha=\{\alpha_t\}$ be an E$_0$-semigroup on a type II$_1$ factor $\m$ with trace $\tau$, acting standardly on $L^2(\m)$, and $\Omega$ be the cyclic and separating vector. For every $t >0$, denote 
$$E^\alpha_t= \{T \in B(L^2(\m)): \alpha_t(m) T = Tm ~ \forall ~m \in \m\},$$
the set of all intertwining operators for $\alpha_t$. It is proved in \cite{alev} that the family $\{E^\alpha_t: t\in (0,\infty)\}$ forms a product system of Hilbert modules over $\m^\prime$. Since our tools are not product system of Hilbert modules in this paper, we do not recall this theory and results from \cite{alev} in full detail. We reproduce the following theorem from \cite{alev}, and make a few definitions and remarks about product systems of Hilbert modules.  All our modules are von Neumann modules, which may be defined as weakly closed (equivalently, strongly closed) subspaces $E\subseteq B(H)$ satisfying $EE^*E\subseteq E$. $E^*E$ is the von Neumann algebra acting on the right and $EE^*$ is the collection of adjointable operators. The inner product is $\ip{x,y}=x^*y$ (see \cite{skiede}, Part I, Chapter 3  or \cite{BMSS} for details).

\begin{thm}\label{alev-module}
$E^\alpha_t$ is full, self-dual $\m^\prime-\m^\prime$ Hilbert bimodule. Its natural $w^*-$topology coincides with the relative $\sigma-$weak topology.
\end{thm}

We can also define $$ E^{\alpha'}_t = \{ T \in B(L^2(\m)): \alpha_t'(m') T = Tm' ~ \forall m' \in \m'\},$$ the complementary product system or the product system corresponding to the complementary \en-semigroup $\alpha'$. $E^{\alpha'}_t$ satisfies all the properties of $E^\alpha_t$ with $\alpha$ replaced with $\alpha'$. In fact $$E^{\alpha'}_t=JE^\alpha_tJ ~\forall ~t>0.$$

Self-dual bimodules are von Neumann modules. Again we do not require the theory of von Neumann bimodules in full detail here. We only use the following  lemma on von Neumann modules in this section. The reader can refer either to \cite{skiede} or to \cite{BMSS} (Proposition 1.7) for a proof. In this lemma $E_1$, $E$ can be considered as only right von Neumann modules. 

\begin{lem}\label{riez}
If $E_1$ is a (strongly closed) $\m$-submodule of a Hilbert von Neumann $\m-$module $E$ and $E_1 \neq E$. Then there exists a non-zero $y \in E$ such that $y^*x=0$ for all $x \in E_1$. 
\end{lem}

The bimodules $E^\alpha_t$ and $E^{\alpha'}_t$ for Clifford flows may be described as follows. This description is used in our computation of $Ind_c(\cdot)$.  Let $H_t= \Gamma_a(L^2((0,t),\k^\C))$. Since the Clifford flow $\{\alpha_t:t\geq 0\}$ is the restriction of the CAR flow on $B(\Gamma_a(L^2((0,\infty),\k^\C)))=B(H)$ to the hyperfinite II$_1$ factor $\mathcal{R}$, it satisfies $$\alpha_t(m)=U_t\left(1_{H_t}\otimes m\right) U_t^*~\forall ~m \in \r.$$ Hence we have $\{T_{\xi_t}:\xi_t\in H_t \}\subseteq E^\alpha_t.$ (Here $U_t$ and $T_{\xi_t}$ are as defined in Remark \ref{additsindex}.) 

Suppose $X_t\in E^\alpha_t$ and $X_t^*T_{\xi_t}=0$ for all $\xi_t \in H_t$, then $$X^*_tT_{\xi_t}\xi=X^*_t (U_t(\xi_t\otimes \xi))=0,~ \forall ~\xi\in H, ~\xi_t \in H_t.$$ This implies that $X_t=0$, thanks to the totality of the set $\{U_t(\xi_t\otimes \xi):\xi\in H, \xi_t \in H_t\}$ in $H$. So by Lemma \ref{riez} we conclude that $E^\alpha_t$ is the Von Neumann (right) $\m'-$Module generated by $\{T_{\xi_t}: \xi_t \in H_t\}$. This is the weak operator closure of right $\m$-linear combinations of $\{T_{\xi_t}: \xi_t \in H_t\}$.

We can also define unitary operator $U_t':H_t\otimes H\mapsto H$ by \begin{align*}U'_t(\xi^t_1\wedge\xi^t_2\cdots \wedge\xi^t_n)\otimes (\xi_1\wedge\xi_2\cdots \wedge\xi_m))&\\=\xi^t_1\wedge\xi^t_2\cdots \wedge\xi^t_n\wedge T_t\xi_1\wedge T_t\xi_2\cdots \wedge T_t\xi_m.& \end{align*}  Then it is easy to verify that $$\alpha'_t(m')=U'_t\left( 1_{H_t}\otimes m'\right){U'_t}^*,$$ and   it follows that $\{T'_{\xi_t}:\xi_t\in H_t \}\subseteq E^{\alpha'}_t,$ where \begin{equation}\label{T'}T'_{\xi_t}(\xi)=U'_t(\xi_t\otimes \xi),~\forall ~\xi \in H.\end{equation} Using an argument similar to  $E^\alpha_t$, thanks to Lemma \ref{riez}, we find  $E^{\alpha'}_t$ is the (right) Von Neumann $\m-$Module generated by $\{T'_{\xi_t}: \xi_t \in H_t\}$.

This can be generalized to any \en-semigroup on a type II$_1$ factor $\m$, which is obtained by restricting an \en-semigroup on $B(L^2(\m)$. In particular the product system of Hilbert modules associated with free flows on $\Phi(\k)$ are described by the product system of Hilbert spaces of the free flow on $B(\Gamma_f(L^2(\R_+; \k^\C))$.

As before $\{S_t\in E^\alpha_t:t\geq 0\}$ denotes the canonical unit for $\{\alpha_t:t\geq 0\}$. We prove the following useful lemma, which will be used in many instances, including our computations of the coupling index for  Clifford flows, even Clifford flows and free flows.

\begin{lem}\label{EtMSt}
$E^\alpha_t= [ \m^\prime S_t ] = {\alpha_t(\m)}^\prime S_t $, where $[\m^\prime S_t ]$
denotes the weak operator closure of $\m^\prime S_t$.
\end{lem}

\begin{proof}
Theorem \ref{alev-module} assures us that  $E^\alpha_t$ is a $\m^\prime-\m'$ Hilbert von Neumann bimodule. Since $S_t \in E^\alpha_t$ it follows that $[ \m^\prime S_t ] \subseteq E^\alpha_t.$ We first verify that $F_t=[ \m^\prime S_t ]$ is a
$\m^\prime$-Hilbert von Neumann submodule of $E^\alpha_t$. Notice that $ F_t^*F_t =
[S_t^*\m^\prime S_t].$ 
We only need to verify $F_t^*F_t=\m^\prime $ (see \cite{BMSS} for details). 
  
For any $m^\prime \in \m'$ we have $$S_t^*m'S_t m = S_t^*m' \alpha_t(m)S_t =mS_t^*m'S_t,$$ and hence
$S_t^*\m'S_t  \subseteq \m^\prime$.
On the other hand $$m' = S_t^*S_t m^\prime = S_t^*\alpha_t^\prime
(m^\prime)S_t,$$  where $\{\alpha_t^\prime\}$ is the complementary \en-semigroup on $\m^\prime$ corresponding to $\{\alpha_t\}$.
So it also follows that $$\m' \subset S_t^*\alpha_t^\prime(\m^\prime)S_t
\subset S_t^*\m'S_t.$$ 
We have verified that $\m' = F_t^*F_t$ and that $F_t$ is a von Neumann submodule of $E^\alpha_t$.

Now suppose there exists a  
$T \in  E^\alpha_t$ and  
$T \perp \m^\prime S_t$ with respect to the $\m'$ valued inner product of  $E^\alpha_t$, i.e., $T^*m^\prime S_t = 0 $ 
for all $m^\prime \in \m^\prime $. 
Then 
$T^*m^\prime S_t\Omega = T^*m^\prime \Omega =0$. Since $T^*$ vanishes on a total set, we conclude that 
$T= 0$. Now we deduce from Lemma \ref{riez} that  
$E^\alpha_t = [\m^\prime S_t]$.

To prove the the second equality, we refer to \cite{alev}(see Proposition 3.4), where it is proved that the space of all adjointable operators $B(E^\alpha_t)$ on the Hilbert von Neumann module $E^\alpha_t$ is equal to ${\alpha_t(\m)}^\prime$. This in particular implies that $E^\alpha_t= {\alpha_t(\m)}^\prime S_t.$   \end{proof}

It is proved in \cite{alev} that $\{E^\alpha_t: t \geq 0\}$ forms a product system of $\m'-\m'$ modules with respect to the tensor product defined by operator multiplication,  and that it is a complete invariant for the \en-semigroup $\alpha$. That is two \en-semigroups $\alpha$ and $\beta$ are cocycle conjugate if and only if the corresponding product system of Hilbert modules $\{E^\alpha_t:t \geq 0\}$ and $\{E^\beta_t: \geq 0\}$ are isomorphic. An isomorphism between product system of Hilbert modules is a family of $\m'-\m'-$linear module isomorphisms, satisfying some measurability conditions, which respects the tensor products structure induced by operator multiplication.  

For a single fixed  \en-semigroup $\alpha$, automorphisms  of $E^\alpha$ are given by the gauge cocycles of $\alpha$. In fact, for every gauge cocycle $U$ of $\alpha$, $U_t$ acts on the left of $E^\alpha_t$ as an adjointable operator since it is contained in  $\alpha_t(\m)'$ (see \cite{alev}) and it is left $\m'-$linear since $\{U_t:t\geq 0\}\subseteq \m$. This gives rise to an automorphism of $E^\alpha$. Conversely, for any given automorphism $\phi_t:E^\alpha_t\mapsto E^\alpha_t$, the proof of 3.12 of \cite{alev} implies that there is a unique $\alpha$-cocycle $\{U_t:t\geq 0\}$ satisfying $U_tT= \phi_t(T)$. Further $$U_t\alpha_t(m) T=U_tTm=\phi_t(T)m=\alpha_t(m)\phi_t(T)=\alpha_t(m)U_tT,$$ for all $T\in E^\alpha_t, m\in \m$ and $t\geq 0$. Now Lemma \ref{EtMSt}   implies that $U_t\in \alpha_t(\m)'$, hence is a gauge cocycle.

\begin{rem}
Our computations in Section \ref{gauge} imply that the product systems of Hilbert modules associated with Clifford flows, even Clifford flows and free flows do not admit any non-trivial automorphisms.  
\end{rem}

\begin{prop}\label{tensors}
If $\alpha$ and $\beta$ \en-semigroups are on II$_1$ factors $\m_1$ and $\m_2$ respectively, then $$E^{\alpha\otimes \beta}_t= E^\alpha_t\otimes E^\beta_t~\forall ~t >0.$$
\end{prop}

\begin{proof}
Clearly for $T \in E^{\alpha}_t$ and $S \in E^{\beta}_t$ $T\otimes S \in E^{\alpha\otimes \beta}_t$. So it follows that $E^{\alpha}_t\otimes E^{\beta}_t \subseteq E^{\alpha\otimes \beta}_t$ for all $t >0$. On the other hand suppose there exists a $X \in E^{\alpha\otimes \beta}_t$ such that $X^*(T\otimes S)=0$ for all $T \in E^{\alpha}_t$ and $S \in E^{\beta}_t$, then $X^*(m_1'\otimes m_2')\Omega=0$ for any $m_1'\in\factor_1'$ and $m_2'\in\factor_2'$, thanks to Lemma \ref{EtMSt}. Hence $X=0$. Now the rest of the proof follows from Lemma \ref{riez}.   \end{proof}

Due to the existence of the canonical unit all \en-semigroups on a type II$_1$ factor are spatial, that is they admit a unit. We can define complete spatiality as follows. 

\begin{defn}\label{I}
An \en-semigroup is said to be type I, or completely spatial, if there exists a subset $S\subset \mathcal{U}_\alpha$ $$E^\alpha_t=[u_{t_1}m'_1u_{t_2}\cdots u_{t_n}m'_n: u_{t_i} \in S, m'_i\in \m', t_1+\cdots+ t_n=t, n \in \N ],$$ where $[\cdot]$ denotes the weak operator closure of the span. Otherwise it is said to be type II.
\end{defn}

We know that the CAR flow on $B(\Gamma_a(L^2((0,\infty),\k^\C)))$ is type I, so we do have enough units in the associated product system of Hilbert spaces $\{H_t\}$. Since the product system of Hilbert modules associated with the Clifford flow,  are the (right) $\m'-$modules generated by $\{T_{\xi_t}: \xi_t\in H_t\}$, it is easy to see that the condition in Definition \ref{I} is satisfied, with $S$ being the set of all units of the CAR flow. This show that Clifford  flows are of type I. 

It is shown in \cite{freesemigroups}, that the free flows on $B(\Gamma_f(L^2((0,\infty), \k^\C))$ are type I, for any $\dim(\k)$. Arguing as for Clifford flows, we get that free flows of any rank on $\Phi(\k)$ are of type I.  In fact, the above argument can be extended to prove that any \en-semigroup on a II$_1$ factor $\m$, which is the restriction of a type I \en-semigroup on $B(L^2(\m))$, is type I.

\begin{prop}\label{EM}
Let $\alpha$ be a  \en-semigroup on a II$_1$ factor $\m$, which is not an automorphism of $\m$.  Then $$E^\alpha_t \cap \m = \emptyset.$$
\end{prop}

\begin{proof}
Fix $t>0$. Suppose $m_t\in E^\alpha_t\cap \m$ is an intertwiner for $\alpha_t$ then $m_t^*m_t\in \m \cap \m'$. Without loss of generality, after dividing by a scalar if needed, we assume that $m_t^*m_t=1$. If $m_t$ is unitary,  we would have $\alpha_t(m) =m_tmm_t^*$ for all $m \in m$. Since $m_t \in \m$, this means $\alpha_t$ is an (inner) automorphism, contradicting our assumption. So $m_t$ is a isometry in $\m$, which is not a unitary. But this is not possible since $\m$ is a finite factor.   \end{proof}

\begin{rem}\label{type I remark}
Here we have defined type I using right $\m'$-linear combinations  of products of units, and we expect our definition to accommodate a rich theory with both type I and type II examples. The other possible definition of type I, using left $\m'$-linear combinations of tensor products of units, is vacuous. Under this definition, all product systems of Hilbert modules associated to \en-semigroups on II$_1$ factors are of type I. This follows immediately from lemma \ref{EtMSt}.

In \cite{tips} a different definition of type I is given. Namely, that the product system is generated by ``continuous units'' (see \cite{tips} for precise definitions). Under this condition no product system of Hilbert modules associated to an \en-semigroup on a \twoone factor is type I. This follows from the fact that type I product systems in \cite{tips} admit a central unital unit, whereas the product systems associated to \en-semigroups on \twoone factors contain no central elements, by Proposition \ref{EM}. In fact the associated CP semigroup of any unit in $E^\alpha_t$ cannot be uniformly continuous, as defined and assumed in the main results of \cite{tips}.
\end{rem}

Now we describe the super product system for an \en-semigroup $\alpha$ on a II$_1$-factor $\m$.  For every $t > 0 $, let $H^\alpha_t = E^\alpha_t \cap E^{\alpha'}_t$. Then the operator norm on $H^\alpha_t$ arises from an inner product and $H^\alpha_t$ is a Hilbert space with respect to that inner product. In fact, for $ S , T \in H^\alpha_t $ and $ m\in M , m' \in M^\prime$, we have: 
\[
T^*Sm= T^* \alpha_t(m)S = (\alpha_t( m^*) T)^*S = (Tm^*)^* S = mT^*S, 
\]
since $S, T \in E^\alpha_t$; on the other hand  we also have 
\[
 T^*Sm'  = T^*  \alpha^\prime_t(m')S = (\alpha^\prime _t( m^{\prime *}) T)^*S = (T m^{\prime *})^*S = m' T^*S, 
\]
since $ S, T \in E^{\alpha'}_t$. Since $M$ is factor, $T^*S$ commutes with all operators in $B(L^2(\m))$. So $T^*S$ is a scalar multiple of the identity and we define $\ip{S,T}_t$ by
$T^*S = \ip{S,T}_t  1.$ 
The operator norm norm on $H^\alpha_t$ coincides with the norm given by this inner product since $$\|T\|^2=\|T^*T\|=\|\ip{T,T}_t1\|=\ip{T,T}_t.$$ Since both $E^\alpha_t$ and $E^{\alpha'}_t$ are closed under operator norm, $H^\alpha_t$ is also closed under operator norm. Hence $H_t^\alpha$ is Hilbert space with respect to this inner product.

The fact that both $E^{\alpha}_s$ and $ E^{\alpha'}_t$ are product systems of Hilbert modules, imply for $X_s\in H^\alpha_s $ and $ X_t\in H^\alpha_t$, that  $X_s X_t \in  H^\alpha_{s+t}$, that is $H^\alpha_s H^\alpha_t\subseteq  H^\alpha_{s+t}.$ Further  for $X_s, Y_s\in H^\alpha_s $ and $ X_t, Y_t\in H^\alpha_t$, \begin{align*}  \ip{X_sX_t,Y_sY_t}_{s+t} 1 &= X_t^*(X_s^*Y_s)Y_t\\ & = \ip{X_s,Y_s}_s X_t^*Y_t\\ & =\ip{X_s,Y_s}_s\ip{X_t,Y_t}_t1.\end{align*}
We have verified that the map
 $U_{s,t} : H_t \otimes H_s \mapsto H_{s+t}$,
given by
$$U_{s,t} ( Y_s \otimes Y_t) = Y_sY_t, ~~Y_s \in H_s, Y_t \in H_t$$ is an isometry. We have the following theorem.

\begin{thm}
Let $\alpha=\{\alpha_t: t\geq 0\}$ be an \en-semigroup on a II$_1$ factor $\m$, then $\{H_t^\alpha, U_{s,t}\}_{s,t \in (0,\infty)}$ is a super product system. 
\end{thm}

\begin{proof}
The associativity axiom follows immediately from the associativity of multiplication of operators. The measurability axiom can be proved in an exactly similar manner to product systems, as given in \cite{arveson} (see Theorem 2.4.7, page 37).   \end{proof}

\begin{rems} $\{U_t: t\geq 0\}$ is a multi-unit for an \en-semigroup $\alpha_t$ on II$_1$ factor if and only if $\{U_t:t\geq 0$ forms a unit for the super product system $(H_t^\alpha,U_{s,t})$. (For semigroups measurability is equivalent to strong continuity.)
\end{rems}

The above described concrete super product systems forms an invariant for \en-semigroups on II$_1$ factors.

\begin{thm}\label{sps invariant}
If two \en-semigroups $\alpha$ and $\beta$ acting on a II$_1 $ factor $\m$ are cocycle conjugate, then the corresponding super product systems $\{H_t^\alpha: t\geq 0\}$ and $\{H_t^\beta: t\geq 0\}$ are isomorphic.
\end{thm}

\begin{proof}
If $\beta$ is replaced with a conjugate version, we get an isomorphic super product system given by the isomorphism $T\mapsto Ad_U(T)$, where $U$ is the unitary operator which implements the conjugacy (see Section \ref{pre}). So without loss of generality we assume $\m$ acts standardly on $L^2(\m)$ and $\alpha$ and $\beta$ are cocycle perturbations of each other. Let $U_t$ be an $\alpha$-cocycle such that $\beta_t(\cdot)=U_t\alpha_t(\cdot)U_t^*$. Then $JU_tJ$ is an $\alpha'-$cocycle. Now define $V_t: H_t^\alpha\mapsto H_t^\beta$ by $$V_t(T)=U_tJU_tJT,~~\forall ~T \in H_t^\alpha.$$ $V_t$ provides the required isomorphism. Indeed for $T \in H_t^\alpha,$ $$\beta_t(m) V_t(T)= U_t\alpha_t(m)U_t^*U_tJU_tJT=U_tJU_tJ \alpha_t(m)T=V_t(T)m,$$ \begin{align*}\beta_t'(m') V_t(T) &= JU_tJ\alpha_t'(m')JU_t^*JJU_tJU_tT\\ & =JU_tJU_t\alpha_t'(m')T=V_t(T)m',\end{align*}  for all $m \in \m, m'\in \m'$. This implies $V_tH_t^\alpha\subseteq H^\beta_t$. By reversing role of $\alpha$ and $\beta$ with $V_t^*$ the opposite inclusion also follows. It is easy to verify $V_t$ is unitary and provides a Borel map between the super product systems. The equation \ref{prodiso} can be checked as follows. For $T_s \in H^\alpha_s, T_t\in H^\beta_t$, \begin{align*} V_{s+t}U^\alpha_{s,t}(T_s \otimes T_t) &= V_{s+t}T_sT_t\\ &= U_{s+t}JU_{s+t}JT_sT_t\\ &=U_s\alpha_s(U_t)JU_sJ\alpha_s'(JU_tJ)T_sT_t\\&= U_sJU_sJT_s U_tJU_tJT_t\\ &=U_{s,t}^\beta(U_sJU_sJT_s \otimes U_tJU_tJT_t)\\ &=
 U_{s,t}^\beta (V_s \otimes V_t)(T_s\otimes T_t)\end{align*} So we do have $(H^\alpha,U^\alpha_{s,t})\cong (H^\beta,U^\beta_{s,t}))$.   \end{proof}


Recall that the $*-$preserving property of $\alpha_t$ implies that $J$ and $S_t$ commute for all $t\geq 0$. Now  Lemma \ref{EtMSt} implies \begin{equation}\label{HEJ} H^\alpha_t = \left(\alpha_t(\m)'\cap J\alpha_t(\m)'J\right)S_t.\end{equation}

\begin{prop}
If $\alpha$ and $\beta$ are two \en-semigroups on II$_1$ factors $\m_1$ and $\m_2$ respectively. Then $$H^\alpha_t\otimes H^\beta_t=H^{\alpha\otimes \beta}_t~\forall ~t>0.$$
\end{prop}

\begin{proof}
Clearly the canonical units satisfy  $S^{\alpha \otimes \beta}_t= S^{\alpha}_t\otimes S^\beta_t$ for all $t>0$. If $J_1$ and $J_2$ are modular conjugation with respect to $\m_1$ and $\m_2$ respectively, then the modular conjugation $J$ for $\m_1\otimes \m_2$ is given by $J_1 \otimes J_2$.  Now, thanks to the relation \ref{HEJ}, we have $H^{\alpha\otimes \beta}_t$ \begin{align*} & =
 \left(((\alpha_t\otimes \beta_t)(\m_1\otimes \m_2))'\cap J((\alpha_t\otimes \beta_t)(\m_1\otimes \m_2))'J\right)S^{\alpha\otimes \beta}_t \\
&= \left((\alpha_t(\m_1)'\otimes \beta_t(\m_2)')\cap (J_1\alpha_t(\m_1)'J_1\otimes J_2\beta_t(\m_2)'J_2)\right)S^{\alpha\otimes \beta}_t \\
&=(\alpha_t(\m_1)'\cap J_1\alpha_t(\m_1)'J_1)S^{\alpha}_t \otimes (\beta_t(\m_2)'\cap J\beta_t(\m_2)'J)S^\beta_t  \\
&=H^\alpha_t \otimes H^\beta_t ~~ \forall ~t>0.\end{align*}   \end{proof}

To facilitate neat computations, we record a simple observation as the following proposition, which realizes the super product systems $(H_t^\alpha,U^\alpha_{s,t})$ as concrete Hilbert subspaces of $H=L^2(\m)$.

\begin{prop}\label{HtOmega}
Let $\alpha$ be an \en-semigroup on a II$_1$ factor $\m$ acting standardly on $H=L^2(\m)$, with cyclic and separating vector $\Omega$.  Then $(H^\alpha_t \Omega, U_{s,t})$ forms a super product system, where $$U_{s,t}(X\Omega \otimes Y\Omega)=XY\Omega, ~~\forall~ X\in H^\alpha_s, Y \in H^\alpha_t.$$  Further $(H^\alpha_t \Omega, U_{s,t})$ is isomorphic to $(H_t^\alpha,U^\alpha_{s,t})$ as a super product system.
\end{prop}

\begin{proof} The association $H^\alpha_t\ni X\mapsto X \Omega \in H_t^\alpha \Omega$ preserves inner products and is onto, hence provides a unitary operator. In fact $$\ip{X\Omega, Y\Omega}=\ip{\Omega, X^*Y\Omega}=\ip{\Omega,\ip{X,Y}\Omega}=\ip{X,Y}.$$   \end{proof}

Since the following useful assertion will be repeatedly used, we separate it as a lemma. Here and elsewhere, $\rho_{m_0} (=Jm_0^*J)$ denotes the bounded operator determined by right multiplication by $m_0$, that is  $$\rho_{m_0}(m\Omega)=mm_0\Omega, ~~\forall ~m\in \m.$$ From Proposition \ref{EtMSt}, we have $$H_t^\alpha=E_t^\alpha\cap E^{\alpha'}_t=[\m S_t]\cap \alpha_t(\m)'S_t.$$
Hence for any  $ A\in H^\alpha_t$ there exists a $T \in \alpha_t(\m)'$ such that $A=TS_t$. 

\begin{lem}\label{rho}
Let $\alpha$ be an \en-semigroup on a II$_1$ factor $\m$. Then for any $A=TS_t \in H^\alpha_t$ we have $$
\rho_{\alpha_t(m)}(T\Omega) = T \alpha_t(m)\Omega,~~ \forall~m \in \m.$$
\end{lem}

\begin{proof}
Since $TS_t \in [\m S_t\Omega]$, there exists a net $\{m_\lambda\}_{\lambda \in \Lambda}\subseteq \m$ such that $m_\lambda S_t$ converges strongly to $TS_t$. Thus the net $\{m_\lambda S_tm\Omega=m_\lambda\alpha_t(m)\Omega\}$ converges in norm to $TS_tm\Omega=T\alpha_t(m)\Omega$ for all $m\in \r$. 
Since $m_\lambda\Omega$ converges to $T\Omega$, we conclude that $\{\rho_{\alpha_t(m)}(m_\lambda\Omega)=m_\lambda \alpha_t(m)\Omega\}$ converges also to $\rho_{\alpha_t(m)}(T\Omega)$. The result follows.   \end{proof}

Now we turn our attention to Clifford flows, our basic examples of \en-semigroups on the hyperfine II$_1$ factor $\r$. 
Set $$H_t^{e,n} =[\xi_1 \wedge\xi_2 \wedge \cdots \wedge \xi_{2m};~ \xi_1, \xi_2 \cdots \xi_{2m} \in L^2((0,t), \k^\C), ~m \in \N_0],$$ for all $t\geq 0$, and $\dim(\k)=n\in \N\cup \{\infty\}$. We may write just $H^e_t$ in many instances when $n$ is arbitrary but fixed.

\medskip

We fix the following notations, when we work with antisymmetric Fock spaces. Pick distinct posets $\Lambda_1$, $\Lambda_2$ order isomorphic to $\N$. Fix an orthonormal basis $\{f_i\}_{i\in \Lambda_1}$ for $L^2((0,t),\k)$ and $\{g_j\}_{j \in \Lambda_2}$ for $ L^2((t,\infty),\k)$ so that $\{f_i\}_{i\in \Lambda_1}\cup\{g_j\}_{j \in \Lambda_2}$ forms an orthonormal basis for $L^2((0,\infty),\k)$. Let $$\p=\{I=(i_1, i_2 \cdots i_m)\in \Lambda_1^m:~1\leq i_1<i_2<\cdots <i_m, m \in \N_0\};$$ $$\F=\{F=(j_1, j_2 \cdots j_m)\in\Lambda_2^m:~1\leq j_1<j_2<\cdots <j_m, m \in \N_0\}.$$ For $I =(i_1, i_2 \cdots i_m)\in \p$, $F =(j_1, j_2 \cdots j_m)\in \F$,  define \begin{align*} u(I) &= u(f_{i_1})u(f_{i_2})\cdots u(f_{i_m});\\ u(F) &=u(g_{j_1})u(g_{j_2})\cdots u(g_{j_m}).\end{align*} Then  it is well known that $\{u(I) u(F)\Omega: I \in \p, F\in \F\}$ forms an orthonormal basis for the antisymmetric Fock space $\Gamma_a(L^2((0,\infty), \k^\C))$ (see for instance \cite{KRP}).

The following proposition describes the super product system for Clifford flows. In the following proposition and elsewhere, an empty wedge product is interpreted as the vacuum vector $\Omega$.

\begin{prop}\label{clifford-prod}
Let $\alpha^n$ be the Clifford flow of rank $n$. Then $H_t^{\alpha^n}\Omega =H^{e,n}_t$ for all $t\geq 0.$ 

Further the unitary map $U_{s,t}:H_s^{\alpha^n}\Omega \otimes H_t^{\alpha^n}\Omega\mapsto H_{s+t}^{\alpha^n}\Omega$, implementing the tensor products in the super product system, is given by 
\begin{align*}  U_{s,t}((\xi_1 \wedge\xi_2 \wedge \cdots \wedge \xi_{2m}) \otimes (\eta_1 \wedge\eta_2  \wedge \cdots \wedge \eta_{2m'}) )&\\ =\xi_1 \wedge\xi_2  \wedge \cdots \wedge \xi_{2m}\wedge S_s \eta_1 \wedge S_s \eta_2\wedge  \cdots \wedge S_s \eta_{2m'}& \end{align*} where $\xi_1, \xi_2\cdots \xi_{2m} \in L^2(0,s)$, $\eta_1, \eta_2 \cdots \eta_{2m} \in L^2(0,t)$.  
\end{prop}

\begin{proof}
Let the operators $T_{\xi_t},T'_{\xi_t}$ be as defined in \ref{T}, \ref{T'}. Through our earlier discussion at the beginning of this section, it is easy to see,  we have $T_{\xi_t}=T'_{\xi_t}$ whenever $\xi_t \in H^e_t$. Consequently $T_{\xi_t} \in H^\alpha_t$, for all $\xi_t \in H^e_t$. We need to prove that  $H^\alpha_t$ contains only these elements. Observe $T_{\xi_t}\Omega=\xi_t$.

Fix $t>0$. 
Let $H^{\alpha^n}_t\ni A=TS_t$, with $T \in \alpha^n_t(\m)'$ be an arbitrary element, then $A\Omega=TS_t\Omega=T\Omega$.
There exists a unique expansion \begin{equation}\label{ONBexp}T\Omega= \sum_{I \in \p, J\in \F}\lambda(I,F)u(I)u(F)\Omega, ~~ \lambda(I,F) \in \C.\end{equation}
Notice that $\alpha^n_t(\r)=\{u(F'): F' \in \F\}''.$ So by lemma \ref{rho} we have \begin{equation}\label{rhou}
\rho_{u(F')}(T\Omega) = T u(F')\Omega,~~ \forall~F' \in \F.
\end{equation}
Since $ T \in \alpha_t^n(\r)'$, we now we use the relation $u(F')Tu(F')=T$ for all $F'\in \F$ as follows. 
\begin{align*}
T\Omega &=u(F')Tu(F')^*\Omega\\
& = u(F') \rho_{u(F')^*}(T\Omega)~ (\mbox{using relation}~\ref{rhou})\\
&= u(F')\rho_{u(F')^*}\left(\sum_{I \in \p, F\in \F}\lambda(I,F)u(I)u(F)\Omega \right)~(\mbox{using}~\ref{ONBexp})\\
&=u(F') \sum_{I \in \p, F\in \F}\lambda(I,F)\rho_{u(F')^*}\left(u(I)u(F)\Omega\right)\\
&= \sum_{I \in \p, F\in \F}\lambda(I,F)u(F') u(I)u(F)u(F')^*\Omega\\
&=  \sum_{I \in \p, F\in \F}\mu_{F'}(I,F) \lambda(I,F)u(I)u(F)\Omega,
\end{align*} where $\mu_{F'}(I,F)=(-1)^{\sigma_{F'}(I,F)}$ with $$\sigma_{F'}(I,F)= |I||F'|+|F||F'|-|F\cap F'|.$$ Since the expansion \ref{ONBexp}  is unique we must have $\sigma_{F'}(I,F)$ is even for all $F'\in \F$.  So we conclude, in the expansion \ref{ONBexp} of $T\Omega$, $\lambda(I,F)=0$ except for the terms indexed by $(I,F)$ satisfying $|I|$ is even and $F$ is empty. Since we started with an arbitrary element of $H^{\alpha^n}_t\Omega$, we have 
\begin{align*}H^{\alpha^n}_t\Omega&=[u(\xi_1)u(\xi_2)\cdots u(\xi_{2m})\Omega:\xi_1, \xi_2 \cdots \xi_{2m} \in L^2((0,t),\k^\C), m \in \N_0]\\
&=[\xi_1 \wedge\xi_2 \wedge \cdots \wedge \xi_{2m};~ \xi_1, \xi_2 \cdots \xi_{2m} \in L^2(0,t), k^\C), ~m \in \N_0].\end{align*}

To prove the remaining assertion, notice, for $H^{\alpha^n}_t\ni A=TS_t$, with $T \in \alpha^n_t(\m)'$, that $$Am\Omega=TS_t m\Omega=\alpha^n_t(m)T\Omega = \alpha^n_t(m)A\Omega.$$ 
Suppose $H^{\alpha^n}_s \Omega\ni A_s\omega=u(\xi_1)u(\xi_2)\cdots u(\xi_{2n})\Omega$ and $H^{\alpha^n}_t \Omega\ni A_t\omega=u(\eta_1)u(\eta_2)\cdots u(\eta_{2m'})\Omega$, then \begin{align*} A_sA_t\Omega & = A_s u(\eta_1)u(\eta_2)\cdots u(\eta_{2m'})\Omega\\
& = \alpha^n_s(u(\eta_1)u(\eta_2)\cdots u(\eta_{2m'}))A_s\Omega\\
&=u(S_s\eta_1)u(S_s\eta_2)\cdots u(S_s\eta_{2m'}) u(\xi_1)u(\xi_2)\cdots u(\xi_{2n})\Omega
\end{align*}

The above computation shows that $$U_{s,t}\left( (u(\xi_1)u(\xi_2)\cdots u(\xi_{2n})\Omega) \otimes (u(\eta_1)u(\eta_2)\cdots u(\eta_{2m'})\Omega)\right)$$ $$=u(S_s\eta_1)u(S_s\eta_2)\cdots u(S_s\eta_{2m'}) u(\xi_1)u(\xi_2)\cdots u(\xi_{2n})\Omega,$$ which consequently implies the remaining assertion of the proposition. The proposition is proved.   \end{proof}

\begin{prop}
The even Clifford flow of rank $n$ has the same super product system as the Clifford flow of index $n$. 
\end{prop}

\begin{proof} Observe that the GNS Hilbert space for even Clifford flow of rank $n$ is the subspace of the antisymmetric Fock space generated by the even particle spaces, which is $H^{e,n}$. Let $\beta^n$ be the even Clifford flow of rank $n$ with $n \in \{1,2,3 \cdots \infty\}$.

We have the isometry $U_t^e: H_t^e \otimes H^e\mapsto H^e$ defined by  \begin{align*}U_t^e(\xi^t_1\wedge\xi^t_2\cdots \wedge\xi^t_{2n})\otimes (\xi_1\wedge\xi_2\cdots \wedge\xi_{2m}))&\\=T_t\xi_1\wedge T_t\xi_2\cdots \wedge T_t\xi_{2m} \wedge \xi^t_1\wedge\xi^t_2\cdots \wedge\xi^t_{2n}.& \end{align*} Direct verification shows, for $\xi^e_t \in H^e_t$, the operator $$T^e_{\xi^e_t}\xi^e =U_t^e(\xi^e_t \otimes \xi^e)$$ defines an intertwiner for $\beta^n_t$. Hence $H^e_t\subseteq H^\beta_t\Omega.$

To show the other way, fix $t>0$ and an arbitrary  $H^{\beta^n}_t\ni A=TS_t$, with $T \in \beta^n_t(M)'$. Then $A\Omega=TS_t\Omega=T\Omega$.
Let $\p_e$ (respectively $\F_e$) consist of the tuples in $\p$ (respectively $\F$) with even length, and $\p_o$ (respectively $\F_o$) the tuples with odd length. 
Then  $$\{u(I_e) u(F_e)\Omega: I_e \in \p_e, F_e\in \F_e\}\cup \{u(I_o) u(F_o)\Omega: I_o \in \p_o, F_o\in \F_o\}$$ forms an orthonormal basis for $H_e$. Therefore  there exists a unique expansion of $T\Omega$ as \begin{equation}\label{ONBexpe}\sum_{I_e \in \p_e, F\in \F_e}\lambda(I_e,F_e)u(I_e)u(F_e)\Omega ~+\sum_{I_o \in \p_o, F_o\in \F_o}\lambda(I_o,F_o)u(I_o)u(F_o)\Omega, \end{equation} with
$\lambda(I_e,F_e), \lambda(I_o,F_o) \in \C.$

Thanks to lemma \ref{rho} and \ref{ONBexpe}, using  the relation $u(F'_e)Tu(F'_e)=T$ for all $F'_e\in \F_e$,   
\begin{align*}
T\Omega &=u(F'_e)Tu(F'_e)^*\Omega\\
& = u(F_e') \rho_{u(F'_e)^*}(T\Omega)\\
&= \sum_{I_e \in \p_e, F_e\in \F_e}\lambda(I_e,F_e)u(F'_e) u(I_e)u(F_e)u(F'_e)^*\Omega\\
&+ \sum_{I_o \in \p_o, F_o\in \F_o}\lambda(I_o,F_o)u(F'_e) u(I_o)u(F_o)u(F'_e)^*\Omega\\
&=  \sum_{I_e \in \p_e, F_e\in \F_e}\mu'_{F'_e}(F_e) \lambda(I_e,F_e)u(I_e)u(F_e)\Omega\\
&+\sum_{I_o \in \p_o, F_o\in \F_o}\mu'_{F'_e}(F_0) \lambda(I_o,F_o)u(I_o)u(F_o)\Omega,
\end{align*} where $\mu_{F'_e}(F_{e/o})=(-1)^{\sigma'_{F'_e}(F_{e/o})}$ with $$\sigma_{F'_e}(F_{e/o})= |F_{e/o}||F'_e|-|F_{e/o}\cap F'_e|.$$ Here we have used the fact that $u(F'_e)$ and $u(I_{e/o})$ commute. Again using the uniqueness of the expansion, forces $\sigma'_{F'_e}(F_{e/o})$ is even for all $F'_e\in \F$.  So we conclude, in the expansion \ref{ONBexp} of $T\Omega$, $\lambda(I_e,F_e)=0$ except for the terms indexed by $(I_e,F_e)$ with $F_e$ empty, and that $\lambda(I_o,F_o)=0$ for all $(I_o,F_o)$. Since we started with an arbitrary element of $H^{\beta^n}_t\Omega$, we have 
$H^{\beta^n}_t\Omega =H^{e,n}_t$. The remaining assertion about the unitary $U_{s,t}$ can be verified in an exactly similar manner as in Proposition \ref{clifford-prod}.   \end{proof}

\begin{rem}\label{not product}
Note that the above super product systems for Clifford flows and even Clifford flows are not product systems.
\end{rem}

From Remark \ref{additsindex} and Corollary \ref{cliffordaddits} the following corollary is immediate.

\begin{cor}
The coupling index of Clifford flows and even Clifford flows is zero for any rank.
\end{cor}

Now we turn our attention to free flows. We fix the following notations for free Fock space. Let  $\{e_i\}_{i\in \N}$ be an orthonormal basis for $L^2((0,\infty),\k)$. Let $$\i=\{I=(i_1, i_2 \cdots i_m): i_l \in \N, 0\leq l\leq m, m \in \N_0\}.$$ (When $m=0$ the tuple is empty.)  
For $I =(i_1, i_2 \cdots i_m)\in \i$,  define \begin{align*} l(I) &= l(f_{i_1})l(f_{i_2})\cdots l(f_{i_m});\\ s(I) &=s(f_{i_1})s(f_{i_2})\cdots s(f_{i_m}).\end{align*} Then $\{l(I)\Omega: I \in \i\}$ forms an orthonormal basis for the free Fock space $H=\Gamma_f(L^2(\R_+; \k^\C))$. (We assume $l(I)\Omega=\Omega =s(I)\Omega$ for the empty tuple $I$.) We call a tuple $I =(i_1, i_2 \cdots i_m)$ as reduced if $i_l\neq i_{l+1}$ for all $1\leq l\leq m-1$. Using induction, it is easy to verify, for any reduced tuple 
$I =(i_1, i_2 \cdots i_m) \in \i$, $i_1\neq i \in \N$, we have 
$$s(I)\Omega=l(I)\Omega;~ s(iI)\Omega=l(iI)\Omega;~s(i^2I)\Omega=l(i^2I)\Omega+l(I)\Omega.$$ (By $(i^rI)$ we denote the concatenated tuple $(i,i,\cdots i,i_1,\cdots i_n)$.) Further 
\begin{align*}
s(i^{2m+1}I)\Omega & =\sum_{r=0}^m k^{2r+1}_{2m+1} l(i^{2r+1}I)\Omega;\\
s(i^{2m+2}I)\Omega & =\sum_{r=0}^m k^{2r}_{2m} l(i^{2r}I)\Omega
\end{align*} for all $n \in \N$. $k^r_m$ are positive integers satisfying relations $k^r_m=0$ if $r>n$ or $r<0$, and $k_{2m+1}^{2r}=0=k_{2m}^{2r-1}$ for all $r$. Further $k_2^2=1=k^1_2$ and the following recursive relation hold $$
k^{2r+1}_{2m+1}=k_{2m}^{2r+2}+k_{2m}^{2r};~ k^{2r}_{2m}=k_{2m+1}^{2r+1}+k_{2m}^{2r-1}, ~\forall ~ 0\leq r\leq m.$$ These relations can be extended to a tuple of the form $(i^{r_1} I_1i^{r_2} I_2\cdots i^{r_p} I_l)$, where $I_l \in \i$ are reduced tuples for $r_l\in \N_0$ $1 \leq l\leq p$. So we have positive integers $k(I,I')$
such that \begin{equation}\label{sl}s(I)\Omega=\sum_{I'\in \i}k(I,I')l(I')\Omega,~\forall ~I \in \i\end{equation} and $k(I, I')$ is zero except for finitely many $I'$ for any fixed $I \in \i$. 

For any reduced tuple $I=(i_1, i_2 \cdots i_m) \in \i$, $i_1\neq i \in \N$, we have  $l(i^2I)\Omega=s(i^2I)\Omega-s(I)\Omega$. Using induction, generalizing to a recursive relation for $l(\cdot)$ in terms of $s(\cdot)$ similarly as above, we can conclude there exists integers $k'(I,I')$ satisfying \begin{equation}\label{ls} l(I')\Omega=\sum_{I\in \i}k'(I',I)s(I)\Omega,~\forall ~I' \in \i.\end{equation}  Again this decomposition is unique and finite. 
We have \begin{equation}\label{kk'}
\sum_{I \in \i} k'(I', I)k(I, I'')=\delta_{I'I''}~ \forall~ I', I'' \in \i.
\end{equation} 
 \begin{equation}\label{k'k}
\sum_{I' \in \i} k(I, I')k'(I', I'')=\delta_{I,I''}~ \forall~ I, I'' \in \i
\end{equation}  

Let $ \Gamma_f^0(L^2((0,\infty), \k^\C))$ is the finite linear span of all finite particle vectors.

\begin{lem}\label{s}
Every $\xi \in \Gamma_f^0(L^2((0,\infty), \k^\C))$ has a unique expansion as $$\xi = \sum_{I \in \i}\mu(I) s(I)\Omega.$$
\end{lem}

\begin{proof}
Existence of the expansion follows, since $\{l(I)\Omega: I \in \i\}$ is an orthogonal basis and from equation \ref{ls}.  For any  $\xi \in \Gamma_f^0(L^2((0,\infty), \k^\C))$ with $\xi=\sum_{I'\in \i}\lambda(I') l(I') \Omega$ we have the decomposition $$\xi=\sum_{I\in \i}\left( \sum_{I' \in \i}\lambda(I')k'(I',I)\right)s(I)\Omega.$$ Since $\xi\in \Gamma_f^0(L^2((0,\infty), \k))$, the length of $I'$ is bounded, and hence any fixed $I$ can future in $k(I', I)$ for only finitely many $I'$. This implies that, in the above expression, $\sum_{I' \in \i}\lambda(I')k'(I',I)$ is a finite sum.

Now suppose $\xi = \sum_{I \in \i} \mu(I)s(I)\Omega$ be any other expansion, then we have\begin{align*}\xi &= \sum_{I \in \i} \mu(I)\left(\sum_{I'\in \i} k(I, I')l(I')\Omega\right)\\
&=\sum_{I'\in \i} \left(\sum_{I\in \i} \mu(I) k(I,I')\right) l(I')\Omega.
\end{align*} This implies $\sum_{I\in \i} \mu(I) k(I,I')=\lambda(I')$ for all $I' \in \i$. (Again note that this is a finite sum.)  Now using relation \ref{k'k}, we have \begin{align*}\sum_{I' \in \i}\lambda(I')k'(I',I) & = \sum_{I' \in I}\sum_{I''\in \i} \mu(I'') k(I,I')k'(I',I'')\\
&=\mu(I) ~\forall ~ I \in \i.
\end{align*} Hence the expansion is unique.   \end{proof}

Let $P_n$ be the projection onto the closed subspace of free Fock space $[l(I)\Omega: I \in \i, |I|\leq n]$.

\begin{lem}\label{pnm}
For every $m \in \Phi(\k)$ there exists an $m_n\in \Phi(\k)$ such that $P_n(m\Omega) =m_n\Omega$. 
\end{lem}

\begin{proof}
Any $s(f_1)s(f_2)\cdots s(f_n)\Omega$ can be expanded as linear combination of $l(f_{i_1})l(f_{i_2})\cdots l(f_{i_l})\Omega$ with ${i_1, i_2\cdots i_l}\subseteq \{1,2 \cdots n\}$ and vice versa for  $l(g_1)l(g_2)\cdots l(g_n)\Omega$. This implies that the lemma holds true for elements of the form $m=s(f_1)s(f_2)\cdots s(f_n)$, and for their linear combinations. Since $$\Phi(\k)=\{s(f_1)s(f_2)\cdots s(f_n): f_1, f_2 \cdots f_n \in L^2((0,\infty),\k), n\in \N\}''$$ to prove the remaining assertions, it is enough if we prove, for any net $\{m_\lambda:\lambda\in \Lambda\}$ satisfying $P_nm_\lambda \Omega=(m_\lambda)_n \Omega$ and $m_\lambda \rightarrow m$ strongly, it follows that $(m_\lambda)_n$ also converges strongly to some $m_n \in \Phi(\k)$ and $P_nm\Omega=m_n \Omega$.

Clearly $(m_\lambda)_n\Omega$ converges to $P_nm\Omega$. Consequently for any arbitrary $m_0 \in \Phi(\k)$ $m_\lambda^nm_0\Omega=\rho_{m_0}((m_\lambda)_n\Omega)$ is convergent. This means, since $\{(m_\lambda)_n: \lambda \in \Lambda\}$ is a bounded net,  it is strongly convergent. Finally if $(m_\lambda)_n \rightarrow m_n \in \Phi(\k)$ strongly, then clearly $P_nm\Omega=m_n \Omega$.   \end{proof}

We can replace $\Phi(\k)$ by its commutant, $s(.)$ by the right multiplication $\rho_{s(.)}$ and by exactly imitating the proof we can arrive at the following Corollary.

\begin{cor}\label{pnm'}
For every $m' \in \Phi(\k)'$ there exists an $m_n'\in \Phi(\k)'$ such that $P_n(m'\Omega) =m_n'\Omega$.
\end{cor}

\begin{lem}\label{Tfinite}
 Let $\gamma$ be the free flow on $\Phi(\k)$. Then for every $\xi \in H^\gamma_t\Omega$ $P_n \xi \in H^\gamma_t\Omega$ for all $n\in \N$ and $t \geq 0$.
\end{lem}

\begin{proof}
We have $H^\gamma_t=[\Phi(\k)S_t]\cap [\Phi(\k)'S_t]$. So, for $TS_t\in H^\alpha_t$,  there exists nets $\{m_\lambda: \lambda\in \Lambda\}\subseteq \Phi(\k)$ $\{m_{\lambda'}': \lambda' \in \Lambda'\} \subseteq \Phi(\k)'$ satisfying $$m_\lambda S_t \rightarrow TS_t~\mbox{and}~ ~m'_{\lambda'} S_t \rightarrow TS_t ~~\mbox{strongly}.$$ Then  both $\{(m_\lambda)_n\Omega\}$ and $\{(m'_{\lambda'})_n\Omega\}$ converges to $T\Omega$. This implies for any $m_0 \in \Phi(\k)$, that $(m_{\lambda})_nS_t m_0\Omega=\rho_{\gamma_t(m_0)}((m_{\lambda})_n\Omega)$ converges to $\rho_{\gamma_t(m_0)}(T\Omega)$. But by
Lemma \ref{rho} we have  $$\rho_{\gamma_t(m_0)}(T\Omega)=T \gamma_t(m_0)\Omega=TS_t m_0\Omega.$$  So we conclude that $(m_{\lambda})_nS_t $ converges strongly to $TS_t$. Replacing appropriately with primes and left action of $\Phi(\k)'$, and by exactly the same reasoning, we get $(m'_{\lambda'})_nS_t $ converges strongly to $TS_t$. So $P_n \xi \in [\Phi(k)S_t]\cap [\Phi(k)'S_t]=H^\gamma_t\Omega$.   \end{proof}

\begin{prop}\label{freesuper}
Let $\gamma$ be free flow on $\Phi(\k)$ of any rank. Then $H^\gamma_t=\C S_t$ for all $t\geq 0$.
\end{prop}

\begin{proof} Clearly $S_t \in H^\gamma_t$ for all $t\geq 0$. For any $\xi \in  H^\gamma_t$ $P_n\xi \rightarrow \xi$ as $n \rightarrow \infty$. So, thanks to Lemma \ref{Tfinite}, it is enough if we prove that $H^\gamma_t\Omega \cap \Gamma^0(L^2((0,\infty), \k^\C)) = \C$. Let $TS_t \in H^\gamma_t$ such that $T\Omega\in  \Gamma^0(L^2(\R_+; \k^\C))$ with $T \in \gamma_t(\Phi(\k))'$ be any arbitrary element.  As before let  $\{e_i\}_{i\in \N}$ be an orthonormal basis for $L^2((0,\infty),\k)$ and also let $\{h_j\}_{j \in \N}\subseteq \{e_i\}_{i\in \N}$ be an orthonormal basis for $ L^2((t,\infty),\k)$.
By Lemma \ref{s}, there exists a unique expansion \begin{equation}\label{sexp}T\Omega= \sum_{I \in \i}\lambda(I)s(I)\Omega, ~~ \lambda(I) \in \C.\end{equation}

Note that $s(h_j) \in \alpha_t(\Phi(\k))$. So by lemma \ref{rho} we have \begin{equation}\label{rhos}
\rho_{s(h_j)}(T\Omega) = T s(h_j)\Omega,~~ \forall~j \in \N.
\end{equation}
Since $ T \in \gamma_t(\Phi(\k))'$, we have $s(h_j)T\Omega=Ts(h_j)\Omega$ for all $j \in \N$. 

Using equations \ref{sexp} and \ref{rhos} we have 
$$\sum_{I \in \i}\lambda(I)s(f_j)s(I)\Omega= \sum_{I \in \i}\lambda(I)s(I)s(f_j)\Omega ~ \forall ~j\in \N.$$ Since this expansion is unique, $\lambda(I)=0$ for any non-empty tuple not ending with $j$. But this is true for all $j$, so $\lambda(I)=0$ for any non-empty $I\in \i$. So $T\Omega =\lambda \Omega $ for some $\lambda \in \C$ and proof of the proposition is over.   \end{proof}


\section{Non cocycle conjugacy}

Even though both the Gauge dimension and coupling index are zero for Clifford flows and even Clifford flows of any index, we still show in this section that both these families contain mutually non cocycle conjugate \en-semigroups. We use the index defined through boundary representation by Alevras (see \cite{powers} and \cite{alev}) and analogues of the $C^*-$semiflows introduced by Remus Floricel in \cite{remus}.

We briefly recall the definition of boundary representation as defined by R. Powers and the index defined by Alveras, which is a conjugacy invariant.  Let $\alpha$ be an \en-semigroup on a II$_1$ factor $\m$  with generator $\delta$, whose domain we denote by $Dom(\delta)$. The generator $-d$  of the canonical unit $\{S_t:t\geq 0\}$ is a maximal skew-symmetric operator whose deficiency space can be identified with the Hilbert space $K=Dom(d^*)/Dom(d)$, with respect to the inner product $\ip{\cdot,\cdot}_0$ defined by $$\ip{[\xi],[\eta]}_0= \frac{1}{2}\ip{d^*\xi,\eta}+\frac{1}{2}\ip{\xi,d^*\eta}.$$ It is shown in \cite{powers} that elements $m\in Dom(\delta)$ leave both $Dom(d^*)$ and $Dom(d)$ invariant, and that the map $$\pi_\alpha(m): Dom(\delta)\mapsto B(K),~~\pi_\alpha(m)[\xi]=[m\xi]$$ is a norm continuous $*-$representation of $Dom(\delta)$. 

If $p \in \pi_\alpha(Dom(\delta))'$ is the largest projection such that the subrepresentation $\pi_\alpha(m)|_{pK}$ is normal, then by extending we get a normal representation of $\m$ on $pK$.  The Powers-Alevras index is defined as the $\m$-dimension of this representation, that is $Ind(\alpha)=\dim_\m(pK).$ 

For Clifford flows $p=1$ and the boundary representation extends to a normal representation of $\r$ on $K=Dom(d^*)/Dom(d)$ (see  \cite{powers} ). Since the Clifford flow of rank $n$ is a restriction of the corresponding CAR flow of rank $n$, and the antisymmetric Fock space is the GNS Hilbert space for $\r$, it follows from \cite{powers} that the Powers-Alevras index for Clifford flow of rank $n$ is also $n$. It is mentioned in \cite{alev} that the index for even Clifford flow of rank $1$ is $1$. We give a proof for general rank $n$.

For a fixed multiplicity $n$, denote the generators of the Clifford flow and even Clifford flow of rank $n$ by $\delta$ and $\delta_e$, and denote the generators of their respective canonical units by $-d$ and $-d_e$. We will denote the boundary representation of the Clifford flow by $\pi$ and of the even Clifford flow by $\pi_e$. For definiteness we write $\r\supset\r_e$ for the factor/subfactor pair given by the Hyperfinite II$_1$ factor generated by even products of $\{u(f):f \in L^2(\R_+;\k)\}$ embedded inside the Hyperfinite II$_1$ factor generated by all products of  $\{u(f):f \in L^2(\R_+;\k)\}$.

\begin{lem}\label{boundary inclusion lemma}
 There is a canonical inclusion of $\dom(d^*_e)/\dom(d_e)$ inside $\dom(d^*)/\dom(d)$ under which the restriction of $\pi$ to $\r_e$ has invariant subspace $\dom(d^*_e)/\dom(d_e)$.
\end{lem}

\begin{proof}
  Since the canonical unit for the Clifford flow respects the decomposition $\Gamma_a(L^2(\R_+;\k))=H_o\oplus H_e$ into odd and even components, its generator splits into a direct sum $d=d_0\oplus d_e$, so there is a canonical inclusion $\dom(d_e^*)\to \dom(d^*)$. This well-defines an inclusion of quotient spaces because $\dom(d_e)=\dom(d)\cap H_e$, so if $\xi_1,\xi_2\in\dom(d_e^*)$ with $\xi_1-\xi_2\in \dom(d)$, then $\xi_1-\xi_2\in\dom(d_e)$. The rest of the Lemma is immediate.   \end{proof}

Note that if $V$ is the inclusion of Lemma \ref{boundary inclusion lemma} then $\pi_e(x)=V^*\pi(x)V,$ so $\pi_e$ extends to a normal representation of $\r_e$.

\begin{prop}\label{boundary index prop}
 For the even Clifford flow of multiplicity $n$ the Powers-Alevras index is $n$.
\end{prop}

\begin{proof}
 Pick an orthonormal basis $e_1,\ldots,e_n$ for $\k$ and write $e(j)=\sqrt{2}e^{-x}\otimes e_j$ for each $j=1,\ldots,n$.
 Recall from \cite{powers} and \cite{alev} that the boundary representation of the Clifford flow decomposes into an orthogonal sum of $n$ standard $\r$-modules, each with an $[e(j)]$ as its cyclic and separating vector. 
 
 Pick a unit vector $f\in\dom(d)$, then under the inclusion of Lemma \ref{boundary inclusion lemma}
 \begin{equation}\label{crucial equation} V\pi_e(\r_e)[f\wedge e(j)]=\pi(\r)[e(j)]\cap V\dom(d^*_e)/\dom(d_e)\end{equation}
 for each $j=1,\ldots,n$. Indeed, the left hand side is clearly contained in the right, whereas
 $$\pi(u(f_1)\cdots u(f_{2n-1}))[e(j)]=\pi(u(f_1)\cdots u(f_{2n-1})u(f))[f\wedge e(j)],$$
 which gives the reverse inclusion. It follows from (\ref{crucial equation}) that the $[f\wedge e(j)]$ generate pairwise orthogonal $\r_e$-modules and these span $\dom(d_e^*)/\dom(d_e)$. Finally,  if  $$\pi_e(x)[f\wedge e(j)]=\pi_e(y)[f\wedge e(j)],$$ for some $x,y\in\r_e$ then $$\pi(x u(f))[e(j)]=\pi(y u(f))[e(j)].$$ So we have $x u(f)=y u(f)$ by the separating property of $[e(j)]$. It follows that $x=y$ and hence $\dom(d^*_e)/\dom(d_e)$ decomposes into $n$ standard $\r_e$-modules.   \end{proof}

For each $t\geq0$ let $\semiflowalg_\alpha(t):=\alpha_t(\m)'\cap\m$. Since these algebras form an increasing filtration, we follow \cite{remus}, and define the inductive limit $C^*$-algebra $\mathcal{A}_\alpha:=\overline{\bigcup_{t\geq0} \mathcal{A}_\alpha(t)}^{\norm{\cdot}}$, together with a semigroup of *-endomorphisms $\alpha|_{\semiflowalg_\alpha}$. This is called the $C^*$-semiflow corresponding to $\alpha$. Since this is a subalgebra of $\m$ there is a canonical trace on $\semiflowalg_\alpha$ which we denote by $\tau_{\alpha}$.

\begin{prop}\label{semiflow prop}
 If $\alpha$ and $\beta$ are cocycle conjugate \en-semigroups then their $C^*$-semiflows are conjugate. Moreover, the intertwining *-isomorphism $\gamma:\semiflowalg_\alpha\to\semiflowalg_\beta$ is implemented by a unitary $U_\gamma:L^2(\semiflowalg_\alpha,\tau_{\alpha})\to L^2(\semiflowalg_\beta,\tau_\beta)$ between the corresponding GNS spaces.
\end{prop}

\begin{proof}
 Without loss of generality we may assume that $\alpha$ is a cocycle perturbation of $\beta$. If $\beta_t=Ad_{U_t}\circ \alpha_t$, for some $\alpha-$cocycle $\{U_t: t \geq 0\}$,  then the *-isomorpism $\gamma:\semiflowalg_\alpha\to\semiflowalg_\beta$ is constructed precisely as in \cite{remus}, Proposition 1.3, by taking the inductive limit of the maps $Ad_{U_t}|\semiflowalg_\alpha(t)\to\semiflowalg_\beta(t)$. Further $\gamma$ intertwines $\alpha|_{\semiflowalg_\alpha}$ and $\beta|_{\semiflowalg_\beta}$. Since each $Ad_{U_t}$ intertwines the corresponding induced traces, the inductive limit $\gamma$ satisfies $\tau_\beta\circ\gamma=\tau_\alpha$. The rest of the proposition follows immediately.   \end{proof}

 We call the triple $(\semiflowalg_\alpha,\alpha|_{\semiflowalg_\alpha},\tau_\alpha)$ the $\tau$-semiflow for $\alpha$. The above lemma shows that two cocycle conjugate \en-semigroups have isomorphic (in the obvious sense of the word) $\tau$-semiflows. For the free flows, the $\tau$-semiflow is trivial, $\semiflowalg_\alpha=\C1$. However, the following proposition shows that, when it is large enough, the triple $(\semiflowalg_\alpha,\alpha|_{\semiflowalg_\alpha},\tau_\alpha)$ says quite a lot about $\alpha$. 

\begin{prop}\label{restriction prop}
 Let $\alpha,\beta$ be \en-semigroups on the \twoone factor $\m$ and suppose that $\semiflowalg_\alpha$ is ultraweakly dense in $\m$. If the $\tau$-semiflow for $\alpha$ is isomorphic to the $\tau$-semiflow of $\beta$ then $\alpha$ is conjugate to a restriction of $\beta$.
\end{prop}

\begin{proof}
 The isomorphism between the semiflows is implemented by a unitary $U_\gamma$ between the respective GNS spaces, as in Proposition \ref{semiflow prop}. Since $\semiflowalg_\alpha$ is ultraweakly dense in $\m$ we have $$L^2(\m,\tau)=L^2(\semiflowalg_\alpha,\tau_\alpha).$$ Now we see that $Ad_{U_\gamma}$ induces an injective $*-$homomorphism $\widetilde{\gamma}:\m\to B(L^2(\semiflowalg_\beta,\tau_\beta))$. It is clear that $\widetilde{\gamma}$ intertwines $\alpha$ with the restriction of $\beta$ to the ultraweak closure of $\semiflowalg_\beta$, as required.   \end{proof}

\begin{cor}\label{density cor}
 Let $\alpha$ and $\beta$ be \en-semigroups on the \twoone factor $\m$ such that both $\semiflowalg_\alpha$ and $\semiflowalg_\beta$ are ultraweakly dense in $\m$. Then $\alpha$ is cocycle conjugate to $\beta$ if and only if $\alpha$ is conjugate to $\beta$.
\end{cor}

\begin{proof}
 If $\alpha$ and $\beta$ are cocycle conjugate then their $\tau$-semiflows are isomorphic. Hence it follows from the proof of Proposition \ref{restriction prop} that $\alpha$ is conjugate to $\beta$.   \end{proof}

\begin{thm}\label{awesome} Even Clifford flows with different rank are not cocycle conjugate. Furthermore Clifford flows with different rank are not cocycle conjugate. \end{thm}

\begin{proof} 
 Pick a  real Hilbert space $\k$ of dimension $n$ and construct the corresponding even Clifford flow $\beta^n$. If $f$ and $g$ are functions in $L^2([0,t];\k)$ then $u(f)u(g)$ belongs to $\beta^n_t(\m)'\cap\m$, hence to $\semiflowalg_{\beta^n}$. But the compactly supported functions are dense in $L^2(\R_+;\k)$, so these elements generate the even Clifford algebra in the strong topology. Thus if $\beta^n$ and $\beta^m$ are cocycle conjugate even Clifford flows, then they satisfy the conditions of Corollary \ref{density cor}, so are conjugate. In particular this implies they have the same boundary index, as defined in \cite{alev}, and hence by Proposition \ref{boundary index prop} they have the same rank, that is $n=m$.
 
If Clifford flows of rank $n$ and $m$ are cocycle conjugate, then their $\tau$-semiflows are isomorphic. It is easily seen that the Clifford flow of rank $n$ (respectively $m$) has the same $\tau$-semiflow as the {even Clifford flow} of rank $n$ (respectively $m$). Since the $\tau$-semiflows are isomorphic, the corresponding even Clifford flows are conjugate, hence by the first part of the theorem $n=m$.   \end{proof}

\begin{rems} 1. The result in Theorem \ref{awesome} is in sharp contrast to the case of reversible flows on the hyperfinite \twoone factor arising from second quantization of bilateral shifts. These are all cocycle conjugate by a result of Kawahigashi (\cite{kawa}).

\noindent 2. If a Clifford flow is cocycle conjugate to an even Clifford flow then they have isomorphic $\tau$-flows, hence it follows that they have the same rank. However, we cannot yet show that the Clifford flow of rank $n$ is not cocycle conjugate to the even Clifford flow of rank $n$.
\end{rems}

\noindent
\tiny{ACKNOWLEDGEMENTS.
The first named author is supported by the UKIERI Research Collaboration Network grant \emph{Quantum Probability, Noncommutative Geometry \& Quantum Information} and the EPSRC, UK. This joint work was initiated during the second named author's visit to Lancaster University, and the visit was supported by the same UKIERI Research Collaboration Network grant.}

\end{document}